\documentclass[11pt,a4paper,leqno]{amsart}
\usepackage{anysize}
\marginsize{3.5cm}{3.5cm}{3cm}{3cm}

\chardef\showllabel=0  
\chardef\refcheck=0    

\usepackage{amsmath,amsthm}
\usepackage{amssymb,esint}
\usepackage{amscd}
\usepackage{mathrsfs}
\usepackage{enumerate}
\usepackage{marginnote}
\usepackage[usenames,dvipsnames,svgnames,table]{xcolor}
\usepackage[]{hyperref}

\hypersetup{
    colorlinks=true,       
    linkcolor=black,          
    citecolor=black,        
    urlcolor=black           
}

\theoremstyle{plain}
\newtheorem{theo}{Theorem}[section]
\newtheorem{prop}[theo]{Proposition}
\newtheorem{lemm}[theo]{Lemma}

\newtheorem{defi}[theo]{Definition}
\theoremstyle{definition}
\newtheorem{rema}[theo]{Remark}

\DeclareMathOperator{\cnx}{div}
\DeclareMathOperator{\diff}{d}

\def\llabel#1{\notag}

\def\dt{\diff \! t}
\def\dx{\diff \! x}
\def\dy{\diff \! y}
\def\dydx{\diff \! y \diff \! x}
\def\defn{\mathrel{:=}}

\def\eps{\varepsilon}

\def\blA{\big\lVert}

\def\brA{\big\lVert}
\def\la{\left\lvert}
\def\lA{\left\lVert}
\def\ra{\right\lvert}
\def\rA{\right\lVert}
\def\le{\leq}

\def\mez{\frac{1}{2}}
\def\ra{\right\rvert}
\def\rA{\right\rVert}
\def\tdm{\frac{3}{2}}
\def\xN{\mathbf{N}}
\def\xR{\mathbf{R}}
\def\xT{\mathbf{T}}
\def\xZ{\mathbf{Z}}

\def\ba{\begin{align}}
\def\be{\begin{equation}}

\def\defn{\mathrel{:=}}

\def\dt{\diff \! t}
\def\dx{\diff \! x}
\def\dy{\diff \! y}
\def\dydx{\diff \! y \diff \! x}

\def\e{\eqref}

\def\ee{\end{equation}}
\def\eps{\varepsilon}

\def\fract{\frac{\diff}{\dt}}

\def\la{\left\vert}
\def\lA{\left\Vert}
\def\le{\leq}
\def\mez{\frac{1}{2}}

\def\ra{\right\vert}
\def\rA{\right\Vert}

\def\xN{\mathbb{N}}
\def\xR{\mathbb{R}}
\def\xT{\mathbb{T}}
\def\xZ{\mathbb{Z}}

\def\partialx{\nabla}

\def\restr#1{\arrowvert_{#1}}
\def\tphi{\tilde\phi}
\def\lec{\lesssim}

\def\colb{\color{black}}
\definecolor{coloraaaa}{rgb}{0.8,0.8,0.8}

\definecolor{coloroooo}{rgb}{0.65,0.3,0.2}

\def\inon#1{\hbox{\ \ \ \ \ \ \ }\hbox{#1}}                

\def\inin#1{\inon{in~$#1$}}

\def\comma{ {\rm ,\qquad{}} }            
\def\supp{\mathop{\rm supp}\nolimits}    
\def\indeq{\qquad{}\!\!\!\!}                     

\numberwithin{equation}{section}

\title{Global-in-time weak solutions for an inviscid free surface fluid-structure problem without damping}

\author{Thomas Alazard}
\address{Thomas Alazard, Universit\'e Paris-Saclay, ENS Paris-Saclay, CNRS, Centre Borelli UMR9010, avenue des Sciences, F-91190 Gif-sur-Yvette France.}
\author{Igor Kukavica}
\address{Igor Kukavica, Department of Mathematics, University of Southern California, Los Angeles, CA 90089}
\author{Amjad Tuffaha}
\address{Amjad, Tuffaha, Department of Mathematics and Statistics, American University of Sharjah, Sharjah, UAE}

\begin{document}

\maketitle

\begin{abstract}
We consider the
Cauchy problem
for an inviscid irrotational fluid on a domain with a free boundary governed
by a fourth order linear elasticity equation.
We first derive the Craig-Sulem-Zakharov formulation of the problem and then
establish the existence
of a global weak solution in two
space dimensions, in the general case without a damping term, for any initial data with finite energy.
\end{abstract}

\section{Introduction}

In this paper, we examine the Cauchy problem associated with the incompressible Euler equations with free surface. Traditionally, the two main models used to investigate the Cauchy problem involve either the constant pressure boundary condition or the surface tension boundary condition, establishing a connection between pressure and the mean curvature of the boundary. 
In this paper, we consider a different type of boundary condition, stemming from the consideration of an elastic interface governed by an Euler-Bernoulli equation.
This problem models the dynamics of an incompressible fluid on a moving domain, with its free boundary evolving according to a fourth-order linear evolution equation forced by the fluid pressure. This describes aeroelastic problems involving inviscid potential flow-structure interactions in the low Mach limit.

This problem has already attracted a lot of attention for viscous equations, assuming in addition that the fluid structure has a damping effect on the free surface elevation. Our main goal is to initiate a study of the case without viscosity and without damping as well. Specifically, even in the absence of any dissipative effects in the problem, we will establish the existence of global-in-time weak solutions in one space dimension. It is noteworthy that we address initial data of any size within the energy space.

Our analysis builds upon established tools developed for studying the Cauchy theory concerning the water-wave equations in Eulerian coordinates. The latter problem has a rich history, dating back notably to the works of Zakharov~\cite{Zakharov1968,Zakharov}, Craig-Sulem~\cite{CrSu}, and Lannes~\cite{LannesJAMS,LannesLivre}, 
who introduced an analysis centered around the Dirichlet-to-Neumann operator. More precisely, we shall use several results from \cite{ABZ1,ABZ3,ABZ-memoir,AM}. We adopt this perspective and consider a formulation that involves a set of equations reduced to the free surface. These equations encompass the displacement function $\eta$ and the trace $\psi$ of the potential flow variable on the free surface. These two variables are connected through the Dirichlet-to-Neumann operator, which describes how the interface moves, and additional terms that consider the dynamics of the interface.

We derive uniform energy estimates 
for Galerkin approximations of the system. The fourth order boundary condition 
allows the propagation of regularity to the interface velocity through the Rellich inequality satisfied by the Dirichlet-to-Neumann operator. Yet, the challenge arises when attempting to handle the limit in the nonlinear term in its classical form. To overcome this, we resort to a novel form 
of the nonlinear term, as introduced by Lokharu~(\cite{LokharuJFM2021}), to study Stokes waves. We will show that this form coincides with the classical form when the solutions are sufficiently regular.

Our approach provides the first existence result of 
global-in-time weak solutions for inviscid free surface flow-structure interaction problem without any damping.

\subsection{The model}
Consider a fluid domain $\Omega$ whose boundary $\Sigma=\partial\Omega$ is a free surface given as the graph of 
some function~$\eta=\eta(t,x)$ which is periodic in $x$, so that at time $t$ we have
\begin{align*}
  \begin{split}
\Omega(t)&=\{ (x,y) \in \xT^d\times \xR\,;\, y < \eta(t,x)\},\\
\Sigma(t)&= \{ (x,y) \in \xT^d\times \xR\,;\, y = \eta(t,x)\},
  \end{split}
  \end{align*}
where $\xT^d=(\xR/2\pi\xZ)^{d}$
is a $d$-dimensional torus. 
We are interested in a problem from fluid mechanics 
where the time evolution of $\Sigma$ is dictated by two equations relating $\eta$ to the fluid velocity field 
$v\colon \Omega\to \xR^{d+1}$ and pressure $P\colon\Omega\to \xR$. 

On the one hand, we assume that the free surface is transported by the flow of a fluid located underneath the free surface. This means that $\eta$ satisfies a kinematic equation
\begin{equation*}
\partial_t \eta = \sqrt{1+ |\nabla \eta |^2} \, v\cdot n \qquad\text{on } \Sigma
,
\end{equation*}
where $n=\frac{1}{\sqrt{1+|\nabla\eta|^2}}\left(\begin{smallmatrix}-\nabla\eta \\ 1\end{smallmatrix}\right)$ is the outward unit 
normal to the free surface. On the other hand, 
we assume that $\eta$ satisfies the plate equation 
\begin{equation}
\partial_t^2\eta +\Delta^2\eta=f,
\label{EQ110}\end{equation}
where the source term $f$ satisfies
\begin{equation*}
f(t,x)=P(t,x,\eta(t,x)),
\end{equation*}
that is, $f$ is given by the evaluation 
at the free surface $\Sigma$ of the pressure of the fluid located underneath~$\Sigma$. 

Then one needs an equation relating $v$ and~$P$. 
This problem has already attracted a lot of attention in the case
where $v$ and $P$ satisfy various equations from fluid dynamics with
dissipation, as is the Navier-Stokes equations.
The latter problem has received a lot of attention. Since it involves different boundary conditions and different notions of solutions, and  since we will use completely different tools, we only refer the reader to the works \cite{MR2166981,MR2349865,MR2644917,MR3466847,GHL,MR3604365,L, LR, MRR, MC1,MR4540753,MC2,MR4039524,MR4450291,schwarzacher-su2023regularity} as well as, for related problems, to the survey papers~\cite{MR2922368,MR4188809}. 
Also, it is worth noting that the plate equation~\e{EQ110} does not incorporate any damping term.

 The goal of the present paper is to consider the case when the
 velocity satisfies the incompressible Euler equations
\begin{equation}\label{WW}
\left\{ \begin{aligned}
&\partial_t v+ v\cdot \nabla_{x,y} v +  \nabla_{x,y} (P +g y)=0 \quad &&\text{in } \Omega, \\[0.5ex]
&\cnx_{x,y} v =0 &&\text{in } \Omega,
\end{aligned} \right. 
\end{equation} 
where 
$g\ge 0$ is gravity constant (in particular 
we can assume that $g=0$) 
and $\nabla_{x,y}=(\nabla,\partial_y)$ with 
$\nabla=(\partial_{x_1},\ldots,\partial_{x_d})$. 
Local-in-time strong solutions for the incompressible Euler equations coupled with elastic moving interface have only been considered recently in~\cite{MR4596023}.
Our aim is to study the existence of global-in-time weak solutions.

We suppose that the flow is irrotational. Then  we have $v=\nabla_{x,y}\phi$ for some 
potential function 
$\phi=\phi(t,x,y)$ satisfying
\begin{equation}\llabel{t5}
\left\{
\begin{aligned}
&\partial_{t} \phi +\mez \la \nabla_{x,y}\phi\ra^2 +P +g y = 0 \quad&&\text{in }\Omega,\\
&\Delta_{x,y}\phi=0\quad&&\text{in }\Omega, 
\end{aligned}
\right.
\end{equation}
where $\Delta_{x,y}=\Delta+\partial_y^2$.

\subsection{Reformulation}
To investigate the existence of global-in-time weak solutions, we will utilize tools developed for the study of the Cauchy problem for the water-waves equation, 
adopting an approach akin to the classical paper of Zakharov~\cite{Zakharov1968,Zakharov}.
We  provide the rest of definitions and 
classical results in Section~\ref{S:2}, ensuring that the paper is self-contained.

Instead of working with the system in the variables $(\eta, \phi)$, we focus on the unknowns $(\eta, \psi)$, where
\begin{equation*}
\psi(t,x) = \phi(t,x,\eta(t,x))
\end{equation*}
represents the trace of $\phi$ on the free surface~$\Sigma$. A key observation is that we can subsequently derive the properties of $\phi$ and $v$, given that the function $\phi$ is harmonic in~$\Omega$.
Importantly, this reformulation reduces the problem to two unknowns, both functions of time $t$ and the horizontal space variable~$x$.

The Dirichlet-to-Neumann operator, denoted by $G(\eta)$, is defined by
\begin{align*}
  \begin{split}
G(\eta)\psi =
\sqrt{1+|\partialx\eta|^2}\,
\partial _n \phi\arrowvert_{y=\eta},
  \end{split}
\end{align*}
where $\partial_n=n\cdot\nabla_{x,y}$
is the normal derivative.
We refer to Section~\ref{S:2} for the rigorous definition.

\begin{prop}\label{IpropN}
Consider two functions $\eta,\psi \in C^\infty(\xT^d)$.
Introduce the linear operator $L$, defined by
\begin{equation*}
L\eta=g \eta+\Delta^2\eta.
\end{equation*}
The fluid-structure problem reads
  \begin{equation}
   \left\{
   \begin{aligned}
   &\partial_t \eta = G(\eta)\psi,\\
   &\partial_t \psi +\partial_t^2\eta+N(\eta,\psi) +L\eta= 0,
   \end{aligned}
   \right.
   \label{IEQ28}
  \end{equation}
where $N(\eta,\psi)$ may be expressed as 
  \begin{equation}\label{I5}
   N(\eta,\psi) =  \frac{1}{2} \vert \nabla_{x} \psi \vert^{2} -\frac{1}{2}  \frac{( G(\eta) \psi + \nabla_{x} \psi \cdot \nabla_{x} \eta)^2}{1+|\nabla\eta|^2}.
  \end{equation}
  \end{prop}

To establish the global existence in time for \e{IEQ28}, we will introduce a Galerkin-type approximation. We aim to prove uniform estimates for the solutions to these approximations by employing a Hamiltonian-type energy estimate. However, it turns out that passing through the limit in the nonlinear term $N(\eta,\psi)$ in its classical form~\e{I5} is not tractable. In response, we adopt an alternative approach, relying on a novel form of the nonlinear term, inspired by the flow force function introduced by Lokharu~(\cite{LokharuJFM2021}) to study Stokes waves. We will prove the following proposition, which gives an 
alternative form of $N(\eta,\psi)$ and states that it coincides with the classical form when solutions exhibit sufficient regularity.

\begin{prop}
\label{IP02}
Let $s>d/2+1$. For all $\eta\in H^{s+1/2}(\xT^d)$ and all 
$\psi\in H^{s}(\xT^d)$, there holds
\begin{equation}
   N(\eta,\psi)=\cnx \int_{-\infty}^\eta \partial_y\phi\nabla \phi\dy
,
   \label{IEQ31}
  \end{equation}
where $\phi=\phi(x,y)$ is the harmonic extension of $\psi$ given by Proposition~\ref{Lemma:decayinfty2} below.
\end{prop}

\subsection{Main theorem}\label{secm}
We now present our main result concerning the existence of global-in-time weak solutions in the two-dimensional case, specifically for $d=1$. First, we introduce a definition of a weak solution for the fluid-structure problem~\eqref{IEQ28}.
\begin{defi}
Let~$T>0$. 
A pair $(\eta,\psi)\in L^\infty((0,T);H^2(\xT)\times H^1(\xT))$ is a weak solution of the problem \eqref{IEQ28} 
on a time interval $[0,T)$
if the following three conditions hold:
\begin{enumerate}[(i)]
\item $\eta \in C([0,T);H^{2-\epsilon}(\xT))$, for every $\epsilon>0$,\\
\item $\psi\in C([0,T);H^{1-\epsilon}(\xT))$, for every $\epsilon>0$,\\
\item the equations \eqref{IEQ28} are satisfied in the sense of distributions
with
  \begin{equation}
   N(\eta,\psi)
   =\cnx \int_{-\infty}^\eta \partial_y\phi \, \partial_{x}\phi\dy
   ,
   \label{IEQ65}
  \end{equation}
where, at a given time $t\in (0,T)$, the function $\phi=\phi(t,x,y)$ is the harmonic extension of $\psi$ given by Proposition~\ref{Lemma:decayinfty2} below.
\end{enumerate}
\end{defi}

\begin{defi}
A pair $(\eta,\psi)\in L^\infty((0,+\infty);H^2(\xT)\times H^1(\xT))$ is a global weak-solution of the problem \eqref{IEQ28} if the above conditions hold on the time interval 
$[0,+\infty)$, or, equivalently, when the restriction to any time interval $[0,T)$ gives a weak solution on~$[0,T)$.
\end{defi}
\begin{rema}
As we will discuss in details later, note that, based on (i) and (ii), we have $\nabla_{x,y}\phi(t)
\in L^2(\Omega(t))$,
uniformly in time,
and thus the expression in \eqref{IEQ65} is well-defined as a distribution
(it is a divergence of an $L^{1}$ function).
\end{rema}

We can now state our main result, which asserts that for any initial data in the energy space, there exists a global-in-time weak solution.

\begin{theo}
\label{T1}
For any initial data $(\eta_0,\psi_0)\in H^2(\xT)\times H^1(\xT)$, 
there exists a global weak solution $(\eta,\psi)$ 
to the system \eqref{IEQ28} satisfying 
$(\eta,\psi)\arrowvert_{t=0}=(\eta_0,\psi_0)$.
\end{theo}
\begin{rema}
Note that $\eta$ and $\psi$ are continuous in the
Sobolev spaces specified in (i) and (ii) so the values of $\eta$ and
$\psi$ at the time~$0$ are well-defined.
\end{rema}
\begin{rema}
Our proof nearly yields a similar outcome for $d=2$. Specifically, we establish the existence and uniqueness of smooth solutions to the Galerkin approximation (refer to Proposition~\ref{T2}; note that for $d=2$, even the solvability of the Galerkin approximation requires a nontrivial trace inequality involving only the ${\rm BMO}$-norm of the slope). Moreover, we could extend this result to $d=2$ by substituting the term involving the bi-Laplacian, $\Delta^2\eta$, with any strictly greater power, that is $\Delta^{2+\epsilon}\eta$ 
with $\eps>0$. However, pursuing this result would entail non trivial
technical problems, and it is not a physical equation. Consequently,
we refrain from exploring this direction in this paper.
\end{rema}

\section{Reformulation of the water-wave problem}\label{S:2}

In this section, we begin by recalling various definitions and 
results on the Dirichlet-to-Neumann operator. Then, we 
will prove Proposition~\ref{IpropN} and Proposition~\ref{IP02} stated in the introduction about the reformulation of the equations.

\subsection{Harmonic extension and the Dirichlet-to-Neumann operator}

We follow the approach
to the classical water-wave problem introduced by Zakharov~(\cite{Zakharov1968,Zakharov}).
As already mentioned, rather than studying 
the system in $(\eta,\phi)$, we work with the pair of unknowns 
$(\eta, \psi)$, 
where
\begin{equation*}
\psi(t,x) = \phi(t,x,\eta(t,x))
\end{equation*}
is the trace of $\phi$ at the free surface~$\Sigma$. 

For better understanding, we recall the definitions and 
classical results starting with smooth functions, and we will later generalize to functions with limited regularity.

We first recall the following elementary result (see~\cite{A-Berkeley}).

\begin{prop}\label{Lemma:decayinfty}
Consider two functions $\eta,\psi \in C^\infty(\xT^d)$, and introduce the domain
\begin{equation*}
\Omega\defn \{\,(x,y)\in\xT^d\times\xR \,;\, y<\eta(x)\,\}.
\end{equation*}
There exists a unique smooth function 
$\phi\in C^\infty(\overline{\Omega})$ such that 
$\nabla_{x,y}\phi\in H^\infty(\Omega)=\bigcap_{k\in\mathbb{N}}H^{k}(\Omega)$,
which satisfies
\begin{equation}
\llabel{defiphi2}
\left\{
\begin{aligned}
&\Delta_{x,y} \phi =0 \quad \text{in}\quad
\Omega,\\
&\phi (x,\eta(x))=\psi(x).
\end{aligned}
\right.
\end{equation} 
Moreover, for any multi-index $\alpha\in\xN^d$ and any $\beta\in \xN$ with $\la\alpha\ra+\beta\geq1$, 
we have
\be\label{decaytozero}
\partial_x^\alpha\partial_y^\beta\phi\in L^2(\Omega)\quad \text{and}\quad
\lim_{y\to-\infty}\sup_{x\in\xT^{d}}\bigl| \partial_x^\alpha\partial_y^\beta\phi(x,y)\bigr|=0.
\ee
\end{prop}

Note that we use the convention $\mathbb{N}=\{0,1,2,\ldots\}$.

We shall say that $\phi$ is the harmonic extension of $\psi$ to~$\Omega$.

\begin{rema}
One can also extend the definition of the 
harmonic extension to functions with limited regularity; see 
Section~\ref{S:4.2}.
\end{rema}

To state the equations that govern the evolution of 
$(\eta,\psi)$, 
we now introduce the Dirichlet-to-Neumann operator. This operator finds applications in various areas of analysis, such as harmonic analysis, inverse problems, and spectral theory, among others.
It also plays a central role 
in the study of the water-wave problem, dating back to the seminal work of Craig and Sulem~\cite{CrSu}. 
By definition, this operator associates with a function defined on the boundary of an open set $\Omega$ the normal derivative of its harmonic extension. In our context, it proves more convenient to introduce the coefficient
$\sqrt{1+|\nabla \eta|^2}$
in the definition.

\begin{defi}\label{C1:defiBV}
Consider two smooth functions $\eta,\psi$ in $C^\infty(\xT^d)$, and denote by $\phi$ the harmonic extension of $\psi$ given by Proposition~\ref{Lemma:decayinfty}.\\
$i)$ 
The Dirichlet-to-Neumann operator, denoted by $G(\eta)$, is 
defined by
\begin{align*}
  \begin{split}
G(\eta)\psi =
\sqrt{1+|\partialx\eta|^2}\,
\partial _n \phi\arrowvert_{y=\eta},
  \end{split}
\end{align*}
where
$\partial_n=n\cdot\nabla_{x,y}$
is the normal derivative.

$ii)$ We also define the operators
\begin{align*}
  \begin{split}
B(\eta)\psi&=(\partial_y\phi)\arrowvert_{y=\eta},\\
V(\eta)\psi&=(\nabla_x \phi)\arrowvert_{y=\eta}.
  \end{split}
  \end{align*}
\end{defi}

\begin{rema}
\label{R01}
$i)$ In the sequel, we shall always use $B$ and $V$ as compact notations for $B(\eta)\psi$ and~$V(\eta)\psi$. 

$ii)$ To clarify the notations, note that
\begin{equation}\llabel{dn}
\begin{aligned}
\sqrt{1+ |\partialx \eta |^2} \,\partial_n \phi \big\arrowvert_{y=\eta(x)}&=
(\partial_y \phi - \partialx \eta \cdot \partialx \phi )\Big\arrowvert_{y=\eta(x)}\\ 
&=(\partial_y \phi)(x,\eta(x))-\partialx \eta (x)\cdot (\partialx \phi)(x,\eta(x)).
\end{aligned}
\end{equation}
\end{rema}

In the next statement, we recall 
two elementary identities satisfied by the quantities $B$ and~$V$.

\begin{lemm}\label{cancellation}
For any $\eta,\psi \in C^\infty(\xT^d)$, we have the identities 
\begin{align}
&B= \frac{G(\eta)\psi+\partialx \eta\cdot\nabla\psi}{1+|\partialx  \eta|^2},\label{defi:BV}\\
&V=\nabla\psi-B \nabla \eta.\label{defi:BV2}
\end{align}
\end{lemm}
\begin{proof}
The chain rule implies 
\begin{equation*}
\nabla \psi=\nabla (\phi(x,\eta(x)))=(\nabla \phi +(\partial_y \phi)\nabla \eta)\arrowvert_{y=\eta}=V+B\nabla \eta
,
\end{equation*}
from where 
we get \eqref{defi:BV2}.
On the other hand, by the definition of the operator $G(\eta)$, we have
\begin{equation*}
G(\eta)\psi=\big( \partial_y \phi-\nabla \eta\cdot \nabla \phi\big)\arrowvert_{y=\eta}=B-V\cdot \nabla \eta,
\end{equation*}
so the identity \e{defi:BV} for $B$  follows from
\eqref{defi:BV2}.
\end{proof}

The expression $G(\eta)\psi$ is linear in $\psi$, but depends nonlinearly on~$\eta$, 
which is the main difficulty in the analysis of free boundary problems. The elucidation of the dependence on $\eta$ is facilitated by the following fundamental identity (we refer to Lannes's original article or his book for more details; see~\cite{LannesJAMS,LannesLivre}). 

\begin{prop}[Lannes's shape derivative formula]\label{P:shape}
For all $\eta,\psi,\zeta\in C^\infty(\xT^d)$, the shape derivative of the Dirichlet-to-Neumann operator is given by
 \be\llabel{n:shape}
  \lim_{\eps\rightarrow 0} \frac{1}{\eps}\bigl( G(\eta+\eps \zeta)\psi -G(\eta)\psi\bigr)
  = -G(\eta)(B\zeta) -\cnx (V\zeta),
 \ee
where $B=B(\eta)\psi$ and $V=V(\eta)\psi$.
\end{prop}
\begin{rema}
This identity extends to functions with limited regularity; however, it is worth noting that the smooth case is the only instance
when we apply the formula.
\end{rema}

The previous result allows us to compute the derivatives of~$G(\eta)\psi$. In particular, we will use the following corollary 
of the shape derivative formula.

\begin{prop}\label{P:t}
For all functions $\eta,\psi\in C^\infty([0,T]\times \xT^d)$, 
the time derivative of the Dirichlet-to-Neumann map is given by 
\begin{equation}
\label{n:time}
\partial_t \left( G(\eta) \psi \right) = G(\eta)( \partial_{t} \psi -B \partial_{t} \eta ) -\cnx( V  \partial_{t} \eta).
  \end{equation}
\end{prop}

\subsection{Functions with limited regularity}\label{S:4.2}
In order to state the next proposition, we recall two notations from \cite{ABZ1},
$\mathscr{D}_0$ and
the space $H^{1,0}(\Omega)$, where 
$\Omega$ is of the form 
$\Omega=\{(x,y): x\in \xT^d,y<\eta(x)\}$ for some Lipschitz function $\eta$.
Denote by~$\mathscr{D}$ the space of functions~$u\in C^{\infty}(\Omega)$ such that
$\nabla_{x,y}u\in L^2(\Omega)$. We then define~$\mathscr{D}_0$ 
as the subspace of functions~$u\in \mathscr{D}$ such that~$u$ is
equal to~$0$ in a neighborhood of~$\partial\Omega$.
Denote by~$H^{1,0}( \Omega)$ the 
space of functions~$u$ on~$\Omega$
such that there exists a sequence~$(u_n)\in \mathscr{D}_0$ for which
$$
u_n  \rightarrow  u \text{ in } L^2( \Omega, (1+|y|)^{-1}\dydx)
,\qquad
\nabla_{x,y} u_n  \rightarrow \nabla_{x,y} u \text{ in } L^2( \Omega, \dydx)
.
$$
Then $H^{1,0}(\Omega)$ is a Hilbert space when equipped with the norm
  \begin{equation*}
   \Vert u\Vert_{H^{1,0}(\Omega)}
   = \Vert \nabla_{x,y}u\Vert_{L^{2}(\Omega)}.
  \end{equation*}

The following statement provides an existence of the harmonic
extension and describes its behavior.
Compared to Proposition~\ref{Lemma:decayinfty}, it assumes that $\eta$ and $\psi$ have a limited regularity; we refer to~\cite[Chapter~$4$]{A-Berkeley} 
for the proof. 

\begin{prop}\label{Lemma:decayinfty2}
Consider a Lipschitz 
function $\eta\in W^{1,\infty}(\xT^d)$, and introduce the domain
\begin{equation*}
\Omega\defn \{\,(x,y)\in\xT^d\times\xR \,;\, y<\eta(x)\,\}.
\end{equation*}
For all $\psi\in H^\mez(\xT^d)$, 
there exists a unique smooth function 
$\phi\in H^{1,0}(\Omega)$, with
\be\label{EQ41a}
\lA \phi\rA_{H^{1,0}(\Omega)}\le C(\lA \nabla\eta\rA_{L^\infty})\lA \psi\rA_{H^\mez}
\ee
and such that 
\begin{equation}\llabel{defiphi1}
\left\{
\begin{aligned}
&\Delta_{x,y} \phi =0 \inin{\Omega},\\
&\phi (x,\eta(x))=\psi(x),
\end{aligned}
\right.
\end{equation}
where the second equation holds in the sense of trace. 

Moreover, for any $h>\lA \eta\rA_{L^\infty}$ and 
for any multi-index $\alpha\in\xN^d$ and any $\beta\in \xN$ with $\la\alpha\ra+\beta\geq1$,
one has
  \begin{equation}
   \partial_x^\alpha\partial_y^\beta\phi\in L^2(\{y<-h\})   
   \label{EQ61}
  \end{equation}
and
$$
\lim_{y\to-\infty}\sup_{x\in\xT^{d}}\bigl| \partial_x^\alpha\partial_y^\beta\phi(x,y)\bigr|=0.
$$
\end{prop}

Note that $\nabla_{x,y}\phi$ belongs only to $L^2(\Omega)$, 
so it is not obvious that one can consider the trace 
$\partial_n \phi\arrowvert_{\partial\Omega}$. However, 
since $\Delta_{x,y}\phi=0$, one can express the normal 
derivative in terms of the tangential derivatives and 
$\sqrt{1+|\nabla \eta|^2}\partial_n \phi\arrowvert_{\partial\Omega}$ 
is well-defined and belongs to~$H^{-\mez}(\xT^d)$. 
As a result, for any $\eta\in W^{1,\infty}(\xT^d)$, 
one can define the Dirichlet-to-Neumann map as an operator satisfying
\begin{equation}
G(\eta)\in \mathcal{L}(H^{\mez}(\xT^d),H^{-\mez}(\xT^d)).
\label{EQ122}\end{equation}

Recall also that $G(\eta)$ is a positive operator, in the sense that, for all $\eta\in W^{1,\infty}(\xT^d)$ and all $\psi\in H^{1/2}(\xT^d)$, 
the divergence theorem implies that
$$
\langle \psi,G(\eta)\psi\rangle_{L^2}
=\int_{\xT^d} \psi G(\eta)\psi\dx=\int_{\Sigma}\phi\partial_n\phi\diff\!\sigma=\iint_{\{y<\eta\}}\la \nabla_{x,y}\phi\ra^2\dydx\ge 0,
$$
where $\phi$ is the harmonic extension of $\psi$.

We shall also use a stronger result. Namely, we recall that the classical trace inequality in a Sobolev space 
(such trace inequalities have been systematically studied by Lannes~\cite{LannesLivre} as well as Leoni and Tice~\cite{MR3989147}):\begin{equation}
\int_{\xT^d}
     \psi G(\eta)\psi \dx\ge C(\lA\nabla\eta\rA_{L^\infty})\lA \psi\rA_{\dot{H}^\mez}^2,
\label{EQ149}\end{equation}
where the homogeneous Sobolev norm is defined by
$$
\lA \psi\rA_{\dot{H}^\mez}=\blA \psi-\hat{\psi}(0)\brA_{H^\mez} \quad\text{with}\quad
\hat \psi(0)
   =
   \frac{1}{(2\pi)^{d}}
   \int_{\xT^d}
   \psi(x) \dx.
$$
We also refer to \cite[Proposition $4.3.2$]{A-Berkeley} (see also~\cite[Appendix $B.1$]{AZ-2023virial}) for the proof of the following refined inequality:
\be\label{N01}
\int_{\xT^d}
     \psi G(\eta)\psi \dx\ge \frac{K}{1+\lA\nabla\eta\rA_{L^\infty}}
     \lA \psi\rA_{\dot{H}^\mez}^2.
\ee
Let us also mention that we will need to use later (cf \e{Ntrace}) a refined version of this trace inequality.

Classical results from functional analysis imply that
if $\eta\in C^\infty(\xT^d)$, then, for any 
$s\ge 1/2$,
the operator $G(\eta)$ is bounded from $H^s(\xT^d)$ into 
$H^{s-1}(\xT^d)$. This result extends to 
the case where $\eta$ 
has limited regularity (see \cite{CN,Gunther-Prokert-SIAM-2006,LannesJAMS,WuInvent}). 
In particular, one has the following estimate 
from~\cite[Theorem $3.12$]{ABZ3}.

\begin{prop}[from~\cite{ABZ3}]\label{P:DN2}
Let~$d\ge 1$,~$s>1+\frac{d}{2}$, and~$\frac 1 2 \leq \sigma \leq s$. 
Then there exists a non-decreasing function~$\mathcal{F}\colon\xR_+\rightarrow\xR_+$ 
such that, for all 
$\eta\in H^{s}(\xT^d)$ and~$f\in H^{\sigma}(\xT^d)$, we have 
$G(\eta)f\in H^{\sigma-1}(\xT^d)$, together with the estimate
 \begin{equation}\label{ests+2}
  \lA G(\eta)f \rA_{H^{\sigma-1}}\le 
  \mathcal{F}\bigl(\| \eta \|_{H^{s}}\bigr)\lA f\rA_{H^{\sigma}}.
 \end{equation}
\end{prop}

The following contraction estimate (which is Theorem~$3.9$ in \cite{ABZ3}) will be very useful when
showing the local existence of solutions 
to the Galerkin systems as well as 
for passing to
the limit. 

\begin{prop}[from~\cite{ABZ3}]
\label{P01}
There exists a non-decreasing
function~$\mathcal{F}\colon\mathbb{R}_{+}\to\mathbb{R}_{+}$
such that
  \begin{equation}
   \lA \bigl( G(\eta_1)-G(\eta_2)\bigr) f\rA_{H^{-\mez}}
   \leq
   \mathcal{F}(\lA (\eta_1,\eta_2)\rA_{W^{1,\infty}})
      \lA \eta_1-\eta_2\rA_{W^{1,\infty}}  \lA f\rA_{H^{\mez}}
    ,
   \label{EQ33}
  \end{equation}
for all~$\eta_1,\eta_2\in W^{1,\infty}(\xT^d)$
and $f\in H^{1/2}(\xT^d)$.
\end{prop}

We also have the following 
contraction estimate from~\cite[Theorem~5.2]{ABZ3}, allowing us to control the difference in a higher-order norm. 

\begin{prop}[from~\cite{ABZ3}]\label{P:DN3}Assume that~$s>1+\frac{d}{2}$. 
There exists a non-decreasing function~$\mathcal{F}$ such that
\begin{equation*}
\lA \left[ G(\eta_1)-G(\eta_2)\right] f\rA_{H^{s-\tdm}}\le 
\mathcal{F} \bigl(\lA (\eta_1,\eta_2)\rA_{H^{s+\mez}}\bigr)\lA \eta_1-\eta_2\rA_{H^{s-\mez}}  \lA f\rA_{
H^s},
\end{equation*}
for all~$\eta_1,\eta_2\in H^{s+\mez}(\xT^d)$ and all~$f\in H^s(\xT^d)$.
\end{prop}

%

Using the Fourier transform, it is easy to check that $G(0)=\la D_x\ra$. Furthermore, it 
is possible to compare $G(\eta)$ and~$\la D_x\ra$ 
thanks to the Bony
paradifferential calculus~\cite{Bony}, using in addition Alinhac's paracomposition operators~\cite{Alipara}. 
The first results in this direction are due to 
Alazard-M\'etivier~\cite{AM}, and 
Alazard-Burq-Zuily~\cite{ABZ3}, following an earlier work 
by~Lannes~\cite{LannesJAMS}. In particular, Proposition~$3.13$ in \cite{ABZ3} asserts that one can compare $G(\eta)$ to an explicit operator for any $\eta$ which is, by the Morrey embedding, in the H\"older space $C^{1+\epsilon}$ for some~$\epsilon>0$. 

\begin{prop}[from~\cite{ABZ3}]\label{coro:paraDN1}
Assume that $d=1$ and consider real numbers $s$, $\sigma$, and $\eps$ such that
\begin{equation*}
s>\frac{3}{2},\qquad \frac{1}{2}\leq \sigma \leq s-\mez,\qquad 
0<\eps\leq \mez, \qquad \eps< s-\frac{3}{2}.
\end{equation*}
Then there exists a non-decreasing function~$\mathcal{F}\colon\xR_+\rightarrow\xR_+$ such that 
\begin{equation*}
R(\eta)f\defn G(\eta)f- \la D_x\ra f
\end{equation*}
satisfies
\begin{equation*}
\lA R(\eta)f\rA_{H^{\sigma-1+\eps}(\xT)}\le \mathcal{F} \bigl(\| \eta \|_{H^{s}(\xT)}\bigr)\lA f\rA_{H^{\sigma}(\xT)}.
\end{equation*}
\end{prop}
\begin{rema}
In \cite{ABZ3}, this result is proven for Sobolev spaces over $\xR$,
but the same proof applies for periodic functions. Also, in \cite{ABZ3},
this result is proven in arbitrary dimension and $R(\eta)$ is defined as 
$R(\eta)f\defn G(\eta)f- T_\lambda f$ where $T_\lambda$ is the paradifferential operator with the symbol $\lambda=((1+|\nabla\eta|^2)|\xi|^2-(\nabla\eta\cdot\xi)^2)^{1/2}$. In the space dimension $d=1$, note that 
$\lambda=|\xi|$, and then $T_\lambda-\la D_x\ra$ is a smoothing operator, 
bounded from $H^s$ to $H^r$ for any $s,r\in\xR$.
\end{rema}

When passing to the limit, we need to obtain the convergence of the
harmonic function $\phi$ using the convergence of $\eta$ and~$\psi$.
To do this, we shall use an estimate comparing the harmonic extensions associated to different functions $\eta$ and~$\psi$. More precisely, 
given $(\eta_k,\psi_k)$, for $k=1,2$, we would like to compare
the two harmonic extensions~$\phi_k$. Since they are defined in different domains, we change 
a variable in the second solution so they are both defined on
$\{z<\eta_1\}$. Let us set
  \begin{align*}
  \begin{split}
   \tilde \phi_2(x,y)
   =
   \phi_2\bigl(x,y-(\eta_1(x)-\eta_2(x))\bigr)
  \end{split}
  \end{align*}
and, for consistency, we also introduce
$
   \tilde \phi_1(x,y)
   =
   \phi_1(x,y)
$.
We choose to work on the domain $\Omega_1\defn \{y<\eta_1\}$, rather than on the flat domain
$\{y<0\}$, so that the leading part of the equation is the Laplacian.

\begin{lemm}
\label{L01}
Consider
  \begin{equation*}
   (\eta_k, \psi_k)
   \in W^{1,\infty}(\xT^d)\times H^{\mez}(\xT^d)
    \comma k=1,2   
  \end{equation*}
with the corresponding functions $\phi_k$ and~$\tilde \phi_k$, for $k=1,2$.
Then, there exists a non-decreasing function $C\colon \xR_+\to\xR_+$
such that
  \begin{align}
  \begin{split}
   \Vert \tilde\phi_k\Vert_{H^{1,0}(\Omega_1)}
   &\le 
   C\big(\Vert ( \nabla\eta_1,\nabla\eta_2)\Vert_{L^{\infty}}\big)
   \Vert \psi_k\Vert_{H^{\mez}}
   ,
  \end{split}
   \label{EQ41}
  \end{align}
for $k=1,2$, 
and
  \begin{align}
  \begin{split}
   \Vert \tilde \phi_1-\tilde \phi_2\Vert_{H^{1,0}(\Omega_1)}
   &\le 
      M
   \Vert \eta_1-\eta_2\Vert_{W^{1,\infty}}
   +M
   \Vert \psi_1-\psi_2\Vert_{H^{\mez}}
  \end{split}
   \label{EQ44}
  \end{align}
where$$
M=C\big(\Vert ( \nabla\eta_1,\nabla\eta_2)\Vert_{L^{\infty}}+\Vert ( \psi_1,\psi_2)\Vert_{H^{\mez}}\big).
$$
\end{lemm}

\begin{proof}
Since $\phi_1=\tilde{\phi}_1$ by notation, when $k=1$, the inequality \e{EQ41} is just~\e{EQ41a}.
For $k=2$, 
set $\eta=\eta_2-\eta_1$. 
In this proof, given a function $f=f(x,y)$, we denote
$$
\tilde{f}(x,y)=f(x,y+\eta(x))
$$
(except of course when $f=\phi_1$). 
To avoid any confusion with notation, recall that $\nabla=\nabla_x$ and $\Delta=\Delta_x$; in other words, these differential operators act only on the $x$ variables. We denote by $\nabla_{x,y}$ and $\Delta_{x,y}$ the operators acting on the full $(x, y)$ variables.

To prove \e{EQ41} when $k=2$, we begin by using a change of variables $(x,y)\mapsto (x,y+\eta(x))$
to get
\be\label{EQ40.5}
\begin{aligned}
\iint_{\Omega_1}\big\vert \widetilde{{(\nabla_{x,y}{\phi_2})}}\big\vert^2\dydx
&=\iint_{\{y<\eta_2\}}\la \nabla_{x,y}\phi_2\ra^2\dydx
\\&
\le  C(\Vert \nabla\eta_2\Vert_{L^{\infty}})
   \Vert \psi_2\Vert_{H^{\mez}}.
\end{aligned}
\ee
Then we use the chain rule to write
\be\label{EQ40.7}
\widetilde{(\partial_y\phi_2)}
=\partial_y \tilde{\phi_2},\quad
\widetilde{(\nabla\phi_2)}
=\nabla\tilde{\phi}_2-\partial_y \tilde{\phi_2}\nabla\eta,
\ee
and by the triangle inequality, we deduce from~\e{EQ40.5} that
 \begin{align}
  \begin{split}
   \Vert \tilde\phi_2\Vert_{H^{1,0}(\Omega_1)}
   &\le 
   C\big(\Vert ( \nabla\eta_1,\nabla\eta_2)\Vert_{L^{\infty}}\big)
   \Vert \psi_2\Vert_{H^{\mez}}.
  \end{split}
   \label{EQ41b}
  \end{align}
Now, by \e{EQ40.7}, we have 
$\partial_y^2\tilde{\phi_2}=\widetilde{(\partial_y^2\phi_2)}$
and 
\begin{align*}
\Delta \tilde{\phi_2}&=\widetilde{(\Delta\phi_2)}
+2\widetilde{({\nabla_x\partial_y\phi_2})}\cdot\nabla\eta
+
\widetilde{({\partial_y\phi_2})}\Delta\eta
+
\widetilde{({\partial_y^2\phi_2})}|\nabla\eta|^2
\\
&=\widetilde{(\Delta\phi_2)}
+2\partial_y\big(\nabla\tilde{\phi_2}-\partial_y\tilde{\phi_2}\nabla\eta)\cdot\nabla\eta+
\partial_y\tilde{\phi_2}\Delta\eta
+\partial_y^2\tilde{\phi_2}|\nabla\eta|^2.
\end{align*}
Next, observe that
$$
\partial_y\tilde{\phi_2}\Delta\eta=
\cnx \big(\partial_y\tilde{\phi_2}\nabla\eta\big)
-\partial_y\big(\nabla\tilde{\phi_2}\cdot\nabla\eta\big)
$$
and then use $\Delta_{x,y}\phi_2=0$ to conclude that
$$
\Delta_{x,y}\tilde{\phi_2}=
\partial_y\Big(\nabla\tilde{\phi_2}\cdot\nabla\eta
-\partial_y\tilde{\phi_2}|\nabla\eta|^2\Big)
+\cnx \big(\partial_y\tilde{\phi_2}\nabla\eta\big).
$$
Setting
$$
f=\partial_y\tilde{\phi_2}\nabla\eta,\quad 
g=\nabla\tilde{\phi_2}\cdot\nabla\eta
-\partial_y\tilde{\phi_2}|\nabla\eta|^2,
$$
the difference $\tilde{\phi}_2-\tilde{\phi}_1$ obeys
  \begin{equation}\llabel{EQ77}
 \left\{
  \begin{aligned}
   &
   \Delta_{x,y} (\tphi_2-\tphi_1)
   =
   \cnx f+\partial_{y} g   \inin{\{y<\eta_1\}},
   \\&
   (   \tphi_2-\tphi_1)
   \restr{y=\eta_1}
   = \psi_2-\psi_1.
  \end{aligned}
  \right.
\end{equation}

Similarly to
Proposition~\ref{Lemma:decayinfty},
given $f, g\in L^{2}(\Omega_1)$,
there exists a 
unique variational solution 
$\phi$ to the problem
$$
-\Delta_{x,y}\phi=\cnx f+\partial_y g \quad\text{in }\Omega_1,\quad 
  \phi\arrowvert_{y=\eta_1}=\psi
  ,
$$
which satisfies the inequality
  \begin{align*}
    \begin{split}
   \Vert \phi\Vert_{H^{1,0}(\Omega_1)}
 \le 
   C(\Vert \nabla\eta_1\Vert_{L^{\infty}})
   (
   \Vert (f,g)\Vert_{L^{2}(\Omega_1)}
   +
   \Vert \psi\Vert_{H^{\mez}}
   )
   .
  \end{split}
  \end{align*}
By combining this estimate with \e{EQ41b}, we immediately get \e{EQ44}, and the lemma is proven.
\end{proof}

\subsection{A Craig-Sulem formulation}

The following statement provides a variant 
of the so-called Craig--Sulem formulation (see~\cite{CrSu,CrSuSu}).
It is derived under the assumption that all the relevant quantities are
smooth.

\begin{prop}\label{propN}
Set
  \begin{equation}\label{N=}
   N(\eta,\psi) = \mez \vert V\vert^2 - \mez B^2 + B(V \cdot\nabla \eta)
   ,
  \end{equation}
and introduce the linear operator $L$ defined by
\begin{equation}
L\eta=g \eta+\Delta^2\eta.
\label{EQ125}
\end{equation}
Then the fluid-structure problem reads
  \begin{equation}
   \left\{
   \begin{aligned}
   &\partial_t \eta = G(\eta)\psi,\\
   &\partial_t \psi +\partial_t^2\eta+N(\eta,\psi) +L\eta= 0.
   \end{aligned}
   \right.
   \label{EQ28}
  \end{equation}
In addition, $N$ may be expressed as 
  \begin{equation}\label{N2=}
   N(\eta,\psi) =  \frac{1}{2} \vert \nabla\psi \vert^{2} -\frac{1}{2} B ( G(\eta) \psi + \nabla \psi \cdot \nabla\eta)
  .
  \end{equation}
  \end{prop}

\begin{proof}
By the definition of $G(\eta)\psi$, 
we get the first equation in \eqref{EQ28} directly from the kinematic boundary condition.

Now, by evaluating  the equation
\begin{equation*}
\partial_t \phi+ \frac{1}{2} \la \nabla_{x,y} \phi \ra^2 +  P +g y =0
\end{equation*}
on the free surface, we obtain
\begin{equation*}
\partial_t\psi-B\partial_t\eta
+\mez |V|^2+\mez B^2+\partial_t^2\eta+L\eta=0.
\end{equation*}
Since $\partial_t \eta=G(\eta)\psi=B-V\cdot\nabla\eta$, it follows that
\begin{equation*}
\partial_t\psi+\partial_t^2\eta +\mez |V|^2-\mez B^2+B (V\cdot \nabla\eta)+L\eta=0,
\end{equation*}
which gives the second equation in~\eqref{EQ28}. The identity \eqref{N2=} can be obtained from \eqref{N=} using the expressions \eqref{defi:BV} and~\eqref{defi:BV2}.
\end{proof}

Here we restate Proposition~\ref{IP02}.

\begin{prop}
\label{P02}
Let $s>d/2+1$. For all $\eta\in H^{s+1/2}(\xT^d)$ and all 
$\psi\in H^{s}(\xT^d)$, there holds
  \begin{equation}
   N(\eta,\psi)=\cnx \int_{-\infty}^\eta \partial_y\phi\nabla \phi\dy
   \label{EQ31}
  \end{equation}
in the $L^{1}(\mathbb{T}^{d})$-sense.
\end{prop}
\begin{rema}Two remarks are in order about the meaning of this identity.

$i)$ Note that the integral in \eqref{EQ31} is well-defined since $\nabla_{x,y}\phi\in L^{2}(\Omega)$.
Also, recall that the divergence is acting in the $x$-variables.

$ii)$ The proof will show that the left-hand side of \eqref{EQ31} belongs
to $L^{1}(\mathbb{T}^{d})$, while the right-hand
side is a derivative of an $L^{1}(\mathbb{T}^{d})$ function.
But then the equality \eqref{EQ31} shows that the right-hand side also belongs to~$L^{1}(\mathbb{T}^{d})$.
\end{rema}

\begin{proof}
The formal computations are as follows. First, we have
  \begin{align}
  \begin{split}
  \cnx\int_{-\infty}^{\eta}\partial_y\phi\nabla \phi\dy
   &=
    \partial_y\phi\arrowvert_{y=\eta}
    \nabla\eta\cdot  \nabla\phi\arrowvert_{y=\eta}
      \\
      &\quad
    + \int_{-\infty}^{\eta}
       (
	  \partial_y\phi
          \Delta \phi
	  +
	  \partial_y\nabla\phi
	  \cdot
	  \nabla\phi
	  )
	  \dy
   ,
  \end{split}
  \label{EQ051}
  \end{align}
where the operators $\nabla$ and $\Delta$ are understood to be in the $x$-variable.
Now, by the definitions of $V$ and $B$, we have
  \begin{equation}
    \partial_y\phi\arrowvert_{y=\eta}
    \nabla\eta\cdot
    \nabla\phi\arrowvert_{y=\eta}
      =B\nabla\eta\cdot V.
   \label{EQ86}
  \end{equation}
On the other hand, by using that $\phi$ is harmonic, we get
\begin{equation*}
  \int_{-\infty}^{\eta}
    \partial_y\phi
    \Delta \phi\dy
   =-\int_{-\infty}^{\eta}
     \partial_y\phi
     \partial_y^2\phi
     \dy
  =-\frac12\int_{-\infty}^{\eta}
     \partial_y ((\partial_{y}\phi)^2)\dy
.
\end{equation*}
Consequently, 
\begin{align}
  \begin{split}
 \int_{-\infty}^{\eta}
   \bigl(
      \partial_y\phi
      \Delta \phi
      +
      \partial_y\nabla\phi
      \cdot
      \nabla\phi
      \bigr)
      \dy
  &=\frac12
         \int_{-\infty}^{\eta} 
         \partial_{y} ( \la\nabla\phi\ra^2-(\partial_{y}\phi)^2)\,\dy
  \\&
  = \frac12 (\la V\ra^2-B^2)
  .
  \end{split}
   \label{EQ60}
\end{align}
Employing \eqref{EQ051}, \eqref{EQ86}, and \eqref{EQ60}, we get
  \begin{align}
  \begin{split}
   &
   \cnx\int_{-\infty}^{\eta}\partial_y\phi\nabla \phi\dy
   =
   B\nabla\eta\cdot V
   +
   \frac12 (\la V\ra^2-B^2)
   ,
  \end{split}
   \label{EQ56}
  \end{align}
and \eqref{EQ31} follows from~\eqref{N=}.

Next, we justify this formal computation for $\eta\in C^{\infty}(\mathbb{T}^{d})$
and $\psi\in C^{\infty}(\mathbb{T}^{d})$.
For $m\in\mathbb{N}$ sufficiently large, let $\theta_m\in C^{\infty}(\mathbb{R})$ be such that
$\theta_m\equiv1$ on $[-2^{m-1},+\infty)$, with
$\supp \theta \in (-2^m,+\infty)$ and
  \begin{equation}
    |\theta_m'(y)|
    \leq
    \frac{1}{2^{m-2}}
\quad\text{for all }y\in \xR.
   \label{EQ54}
  \end{equation}
Then, as in \eqref{EQ051} and \eqref{EQ60}, but including the cutoff $\theta$, we have 
  \begin{align}
  \begin{split}
  &
  \cnx\int_{-\infty}^{\eta}\partial_y\phi\nabla \phi \theta_m\dy
   \\&\indeq
   =
      \bigl(\partial_y\phi
      \nabla\eta
      \cdot
      \nabla\phi
      \bigr)\arrowvert_{y=\eta}
    + \int_{-\infty}^{\eta}
    \bigl(
      \partial_y\phi
      \Delta \phi 
        +
	\partial_y\nabla\phi
	\cdot
	\nabla\phi
     \bigr) \theta_m
   \\&\indeq
   =
      \bigl(\partial_y\phi
      \nabla\eta
      \cdot
      \nabla\phi
      \bigr)\arrowvert_{y=\eta}
      +\frac12
         \int_{-\infty}^{\eta} 
         \partial_{y} ( \la\nabla\phi\ra^2-(\partial_{y}\phi)^2)\theta_m\dy
    ,
    \end{split}
   \label{EQ55}
  \end{align}
where $\theta_m=\theta_m(y)$ is understood.
Using integration by parts, the second term on the far-right side equals
  \begin{align*}
  \begin{split}
   &
   \frac12
         \int_{-\infty}^{\eta} 
         \partial_{y} ( \la\nabla\phi\ra^2-(\partial_{y}\phi)^2)\theta_m\dy
  \\&\indeq
  =
  \frac12 (|V|^2-B^2)
  -
   \frac12
         \int_{-2^{m}}^{-2^{m-1}}
         (\la\nabla\phi\ra^2-(\partial_y \phi)^2)\theta_m'\dy
   ,
  \end{split}
    \end{align*}
for $m$ sufficiently large, depending on~$\Vert \eta\Vert_{L^\infty}$. By \eqref{EQ54}, as $m\to\infty$, the second term converges
to 
$0$,
using $ ( \la\nabla\phi\ra^2-(\partial_y \phi)^2)\in L_{x,y}^{1}$
due to~\eqref{EQ61}.
Now, we consider the left side of \eqref{EQ55}, where we first note that
we have 
$\partial_y \phi(x,\cdot) \nabla \phi(x,\cdot)\in L^1_y$ and hence
  \begin{equation*}
   \int_{-\infty}^{\eta(x)}\partial_y\phi(x,y)\nabla \phi(x,y) \theta_m(y)\dy
   \to    \int_{-\infty}^{\eta(x)}\partial_y\phi(x,y)\nabla \phi(x,y) \dy
   \text{~as $m\to\infty$}
   ,
  \end{equation*}
for almost all $x\in \xT^{d}$,
by the Dominated Convergence Theorem. Moreover, 
the difference satisfies
  \begin{align}
  \begin{split}
   &
   \left\Vert
     \int_{-\infty}^{\eta}\partial_y\phi\nabla \phi \theta_m\dy
       -
    \int_{-\infty}^{\eta}\partial_y\phi\nabla \phi \dy       
   \right\Vert_{L^{1}_x}
   =
   \left\Vert
     \int_{-\infty}^{\eta}\partial_y\phi\nabla \phi (1-\theta_m)\dy
   \right\Vert_{L^{1}_x}
  \\&\indeq
  \leq
    \iint_{-\infty}^{-2^{m-1}}|\partial_y\phi|   |\nabla \phi| \dy\dx
    ,
  \end{split}
   \label{EQ58}
  \end{align}
for $m$ sufficiently large, depending on~$\lA \eta\rA_{L^\infty}$.
Note that the expression on the far-right side of \eqref{EQ58} converges to $0$
as $m\to\infty$, which implies that
  \begin{equation*}
    \cnx
     \int_{-\infty}^{\eta}\partial_y\phi\nabla \phi \theta_m\dy
     \to
    \cnx
     \int_{-\infty}^{\eta}\partial_y\phi\nabla \phi \dy
  \end{equation*}
in the sense of derivatives of $L^{1}(\mathbb{T}^{d})$ functions.
From \eqref{EQ55}, we then obtain
\eqref{EQ56}, showing
that \eqref{EQ31} holds for smooth $\eta$ and~$\psi$.

Finally, let $\eta$ and $\psi$ be as in the statement.
Then we approximate
them with smooth $\eta_m$ and $\psi_m$, where~$m\in\mathbb{N}$.
For each $m\in\mathbb{N}$, consider the corresponding harmonic
extension
$\phi_m$, defined on $\{y<\eta_m\}$,
and the associated function
$   \Phi_m(x,y)
   =
   \phi_m(x,y+\eta_m(x))
$, defined on a fixed domain $\{y<0\}$.
Let $B_m$ and $V_m$ be the corresponding functions as in Remark~\ref{R01}(i).

By the first part of the proof, we have
  \begin{align}
  \begin{split}
   &
   \cnx\int_{-\infty}^{\eta_m}\partial_y\phi_m\nabla \phi_m \dy
   =
   \frac12(|V_m|^2- B_m^2)
   + B_m(V_m \cdot\nabla \eta_m)
   ,
  \end{split}
   \llabel{EQ63}
  \end{align}
which may be rewritten as
  \begin{align}
  \begin{split}
   &
   \cnx\int_{-\infty}^{0}\partial_y\Phi_m(\nabla \Phi_m
   -\partial_{y}\Phi_m \nabla \eta_m) \dy
   =
   \frac12(|V_m|^2- B_m^2)
   + B_m(V_m \cdot\nabla \eta_m)
   .
  \end{split}
   \label{EQ88}
  \end{align}
Now, denoting
$\Phi(x,y)=\phi(x,y+\eta(x))$,
  \begin{align}
  \begin{split}
   &
   \left\Vert
   \int_{-\infty}^{0}
   (
    \partial_{y}\Phi_{m}   \nabla\Phi_{m} 
     - \partial_{y}\Phi\nabla\Phi 
   )
   \right\Vert_{L^1_x}
   \\&\indeq
   \le
   \iint_{\{y<0\}}
   | \nabla\Phi_{m}|
   |
   \partial_{y}\Phi_{m}
     -  \partial_{y}\Phi
   |
   +
   \iint_{\{y<0\}}
   |
   \nabla\Phi_{m} 
     - \nabla\Phi
   |
   |   \partial_{y}\Phi|
  \\&\indeq
   \le
   \Vert \nabla\Phi_{m}\Vert_{L^2_{x,y}}
   \Vert
   \partial_{y}\Phi_{m}
     -  \partial_{y}\Phi
   \Vert_{L^2_{x,y}}
   +
   \Vert
   \nabla\Phi_{m} 
     - \nabla\Phi
   \Vert_{L^2_{x,y}}
   \Vert   \partial_{y}\Phi\Vert_{L^2_{x,y}}
   .
  \end{split}
  \label{EQ51}
  \end{align}
Note that
$   \Vert
   \partial_{y}\Phi_{m}
     -  \partial_{y}\Phi
   \Vert_{L^2_{x,y}}   \to 0
$, as a consequence of Lemma~\ref{L01},
and the first term on the far-right side of \eqref{EQ51} converges to~$0$. Similarly, the second term on the far-right side
of \eqref{EQ51} converges to~$0$.
Also, regarding the second component of the integral in \eqref{EQ88}, we have
  \begin{align}
  \begin{split}
   &
   \left\Vert
   \int_{-\infty}^{0}
   \bigl(
   (\partial_{y}\Phi_m)^2 \nabla \eta_m
     -   (\partial_{y}\Phi)^2 \nabla \eta
   \bigr)
   \right\Vert_{L^1_x}
   \to 0
   ,
   \label{EQ87}
  \end{split}
  \end{align}
as $m\to \infty$,
omitting the proof since it is analogous
to the argument in~\eqref{EQ51}.
Now, we claim that
  \begin{equation}
   \frac12 (|V_m|^2-B_{m}^2)
   + B_m(V_m \cdot\nabla \eta_m)
   \to
   \frac12 (|V|^2-B^2)
   + B(V \cdot\nabla \eta)
   ,
   \label{EQ64}
  \end{equation}
as $m\to\infty$,
in~$L^{1}(\mathbb{T}^{d})$.
For the $B$~part, we
apply
Proposition~\ref{P:DN3}
and obtain,
since $s>d/2+1\geq 3/2$,
  \begin{align}
  \begin{split}
   &
   \Vert (G(\eta_m)-G(\eta))\psi_m\Vert_{L^2}
   \leq
   \mathcal{F} \bigl(\lA (\eta_m,\eta)\rA_{H^{s+1/2}}\bigr)
   \lA \eta_m-\eta\rA_{H^{s-1/2}}
   \lA \psi_m\rA_{H^{s}}
   ,
  \end{split}
   \llabel{EQ67}
  \end{align}
while, by Proposition~\ref{P:DN2},
  \begin{equation}
   \Vert G(\eta)(\psi_m-\psi)\Vert_{L^2}
   \leq
   \mathcal{F} \bigl(\lA \eta\rA_{H^{s}}\bigr)
   \lA \psi_m-\psi\rA_{H^{s}}
   .
  \llabel{EQ71}
  \end{equation}
Also, one has the elementary inequalities
  \begin{align}
   \begin{split}
   \Vert \nabla \eta_m \cdot \nabla \psi_m
        -\nabla \eta \cdot \nabla \psi_m \Vert_{L^2}
   &\le 
   \Vert \nabla \eta_m
        -\nabla \eta  \Vert_{L^\infty}
   \Vert \nabla \psi_m \Vert_{L^2}
   \\
   &
   \lec
   \Vert \eta_m    -\eta  \Vert_{H^{s}}
   \Vert \psi_m \Vert_{H^{1}}
  \end{split}
   \llabel{EQ09}
  \end{align}
and
\be
  \begin{aligned}\llabel{EQ70}
  \Vert \nabla \eta \cdot \nabla \psi_m
        -\nabla \eta \cdot \nabla \psi \Vert_{L^2}
  & \le
   \Vert \nabla \eta  \Vert_{L^\infty}
   \Vert \nabla \psi_m -\nabla \psi\Vert_{L^2}
   \\
   &\lec
   \Vert \eta  \Vert_{H^{s}}
   \Vert \psi_m -\psi\Vert_{H^{1}}
   ,
    \end{aligned}
    \ee
using the embedding $H^{s}\subseteq W^{1,\infty}$.
This show that
$   \Vert \nabla \eta_m \cdot \nabla \psi_m
        -\nabla \eta \cdot \nabla \psi \Vert_{L^2}
\to 0
$ as $m\to\infty$. An analogous derivation shows that
$   \Vert \nabla \eta_m \cdot \nabla \eta_m
        -\nabla \eta \cdot \nabla \eta \Vert_{L^2}
\to 0
$ as $m\to\infty$.
Using the above inequalities
with \eqref{defi:BV},
we get 
  \begin{equation}
   B_m\to B
    \text{~in~} L^{2}
   .
   \label{EQ72}
  \end{equation}
On the other hand, by \eqref{defi:BV2},
  \begin{align*}
  \begin{split}
   \Vert V-V_m\Vert_{L^2}
   &\leq
   \Vert \nabla(\psi-\psi_m)\Vert_{L^2}
   +
  \Vert B-B_m\Vert_{L^2}
  \Vert \nabla\eta_m\Vert_{L^\infty}
   \\&\indeq
   +
  \Vert B\Vert_{L^2}
  \Vert \nabla(\eta-\eta_m)\Vert_{L^\infty}
  ,
  \end{split}
  \end{align*}
which, together with \eqref{EQ72}, easily implies
  \begin{equation}
   V_m\to V
    \text{~in~} L^{2}
   .
   \label{EQ75}
  \end{equation}

Also, using \eqref{EQ72}, \eqref{EQ75},
and $\nabla \eta_m\to \nabla \eta$ in $L^{\infty}$,
due to $H^{s}\subseteq W^{1,\infty}$,
leads to
$B_m(V_m\cdot\nabla \eta_m)\to
 B V\cdot\nabla \eta$ in $L^{1}(\mathbb{T}^{d})$.
Together
with \eqref{EQ72} and \eqref{EQ75}, we then obtain~\eqref{EQ64}.
Finally,
from \eqref{EQ88}, \eqref{EQ51}, \eqref{EQ87},
and \eqref{EQ64}, 
we conclude
  \begin{equation*}
   \cnx\int_{-\infty}^{0}
     \partial_y\Phi(\nabla \Phi-\partial_{y}\Phi \nabla \eta) \dy
   =
   \frac12(|V|^2- B^2)
   + B(V \cdot\nabla \eta)
   ,
  \end{equation*}
which leads to \eqref{EQ31} after a change of variables.
\end{proof}

\section{Galerkin approximations}\label{S:approximate}
Consider initial data with mean value zero:
\begin{equation}
   \int_{\mathbb{T}^d} \eta_0\dx
   =\int_{\mathbb{T}^d} \psi_0\dx
   =0
   .
   \label{IEQ90}
  \end{equation}
We can do so without loss of generality since the mean values of $\eta$ and $\psi$ are preserved (see~\eqref{EQ14} and \eqref{EQ15} below).

We shall define solutions as limits of the solutions to well-chosen approximate systems. 
To define such approximate systems, we use a 
version of Galerkin's method based on Friedrichs mollifiers. To do so, denote by $\hat{u}$ the Fourier transform of a function $u$,
and consider, for $n\in \xN\setminus\{0\}$, the operators $J_n$ defined by 
\begin{alignat*}{3}
\widehat{J_n u}(\xi)&=\hat{u}(\xi) \quad &&\text{for} &&\la \xi\ra\le n,\\
\widehat{J_n u}(\xi)&=0 \quad &&\text{for} &&\la \xi\ra> n,
\end{alignat*}
where we consider
  \begin{equation}\label{def:F}
   \hat f(\xi)
   =
   \frac{1}{(2\pi)^{d}}
   \int_{\xT^d}
   f(x) e^{-i \xi \cdot x}
   \dx
    \comma \xi\in\xZ^d
    .
   \end{equation}
Note that $J_n$ is a projection, i.e., $J_n^2=J_n$; also,
$J_0 u = (2\pi)^{-d}\int u$.
Introduce the space
\begin{equation*}
L^2_n(\xT^d)=J_n \big(L^2(\xT^d)\big)=\{u\in L^2(\xT^d)\,;\,  \hat{u}(\xi)=0 \text{ for all }\la\xi\ra>n\}.
\end{equation*}

Now, we consider the following approximate Cauchy problem:
\begin{equation}\label{A1}
\left\{
\begin{aligned}
&\partial_t \eta = J_n G(\eta)\psi,\\[1ex]
&\partial_t \psi +\partial_t^2\eta+ J_n\big( N(\eta,\psi) 
+L\eta\big)= 0,\\[1ex]
& (\eta,\psi)\arrowvert_{t=0}=J_n (\eta_0,\psi_0).
 \end{aligned}
 \right.
 \end{equation}

The following result states that, for each $n\in\xN\setminus\{0\}$, 
the Cauchy problem~\e{A1} is globally well-posed.

\begin{prop}\label{T2}
Assume that $d=1$ or $d=2$, and consider $n\in \xN\setminus \{0\}$.
For all $(\eta_0,\psi_0)\in L^{2}(\xT^d)^2$ satisfying~\e{IEQ90},
the Cauchy problem~\e{A1} has a unique global-in-time solution
\begin{equation*}
(\eta_n,\psi_n)\in C^{\infty}\big([0,+\infty);L^2_n(\xT^d)^2\big).
\end{equation*}
Moreover,
  \begin{equation}
   \mathcal{E}(\eta_n(t),\psi_n(t))=
   \mathcal{E}(\eta_0,\psi_0)
  ,
   \label{EQ92}
  \end{equation}
for all $t\geq0$, where $\mathcal{E}(\eta,\psi)$ is defined by  \begin{equation}
   2  \mathcal{E}(\eta,\psi)
     =
     \int_{\xT^d}
     \psi G(\eta)\psi \dx
       + \int_{\xT^d} (\partial_{t}\eta)^{2} \dx
       + \int_{\xT^d} (\Delta\eta )^{2} \dx
       + g \int_{\xT^d} \eta^{2} \dx
       .
   \label{EQ91}
  \end{equation}
\end{prop}
\begin{rema}
Although our main result is for the case $d=1$, we prove this proposition also for $d=2$ since the latter is of independent interest. Also, the case $d=2$ requires a subtle improved trace inequality (see~\e{Ntrace} below) 
for which the case $d=2$ is critical. In particular, our analysis does not extend to the (unphysical) case $d\ge 3$.
\end{rema}

The proof consists of four steps. The proof of local-in-time existence is provided in Steps~$1$ and~$2$ below, while
the conservation of energy \eqref{EQ92} is shown in Step~$3$. In the last step we show that the energy conservation implies that the solutions are in fact global in time; the proof also relies on a delicate trace inequality.

\subsection*{Step 1 : an augmented system}
One of the difficulties we face is that the system is not of order~$1$ in time. 
To overcome this, we shall work with an augmented system obtained 
by introducing $\theta=\partial_t \eta$ as an additional unknown. 

We begin by performing formal computations to derive an equation for $\theta$. To do so, write
\begin{equation}
\partial_t \theta=\partial_t (J_nG(\eta)\psi)=J_nG(\eta)(\partial_t \psi-B\partial_t \eta)-J_n\cnx(V\partial_t \eta),
\label{EQ138}
\end{equation}
which follows from the time derivative formula~\e{n:time} for the Dirichlet-to-Neumann operator.

Given $\eta$ and $\psi$, and thus $V=V(\eta)\psi$ and $B=B(\eta)\psi$ as defined by \e{defi:BV} and~\e{defi:BV2},
introduce the linear operator $\Lambda_n(\eta,\psi)\colon C^\infty(\xT)\to C^\infty(\xT)$, defined by
\begin{equation}
\Lambda_n(\eta,\psi)u=J_nG(\eta)(B u)+J_n\cnx(V u).
\label{EQ02}
\end{equation}
It follows from \eqref{EQ138} and \eqref{EQ02} that
$$
\Lambda_n(\eta,\psi)\partial_{t}\eta
-J_n G(\eta)\partial_{t}\psi+\partial_{t}\theta=0.
$$

The triplet $f=(\eta,\psi,\theta)^{T}$ satisfies 
a system of three scalar evolution equations of the form
\begin{equation}
A_n(f)\partial_t f=F_n(f)
\label{EQ139}
,
\end{equation}
where  
\begin{equation*}
A_n(f)=
\begin{pmatrix}
I & 0 & 0 \\
0 & I & I \\
\Lambda_n(\eta,\psi) & -J_nG(\eta) & I 
\end{pmatrix},\qquad 
F_n(f)=\begin{pmatrix}J_nG(\eta )\psi \\[1ex]
-J_nL\eta -J_n N(\eta,\psi)\\ 0
\end{pmatrix},
\end{equation*}
where $I$ is the identity operator.

Our strategy is to prove that this system has a global-in-time solution, and that if $f=(\eta,\psi,\theta)^T$ solves 
this system, then $(\eta,\psi)$ solves the original system. 

In order to invert the matrix $A_n(f)$, we 
introduce 
\begin{equation*}
\tilde A_n(f)=\begin{pmatrix}
I+J_nG(\eta) & 0 & 0 \\
\Lambda_n(\eta,\psi)  & I & -I \\
-\Lambda_n(\eta,\psi) & J_nG(\eta) & I
\end{pmatrix}.
\end{equation*}
Then note that
\begin{equation*}
\tilde A_n(f)A_n(f)= (I+J_nG(\eta))I_3
,
\end{equation*}
where $I_3=\left(\begin{smallmatrix} I & 0 & 0 \\ 0 & I & 0\\ 0 & 0& I\end{smallmatrix}\right)$ is the identity matrix.
Denote
$$
R_n(\eta)= (I+J_nG(\eta))^{-1},
$$
assuming for a moment that one can invert the operator 
$I+J_nG(\eta)$ on $L^2_n(\xT^d)$.
Then observe that \e{EQ139} is equivalent to
\begin{equation}
\partial_t f=\mathcal{F}_n(f)\quad\text{where}\quad \mathcal{F}_n(f)=R_n(\eta)\bigl(\tilde A_n(f)(F_n(f))\bigr).
\label{EQ144}\end{equation}

We now have to justify that one can invert 
the operator $I+J_nG(\eta)$.

\begin{lemm}\label{L:Rn}
For all $\eta\in C^\infty(\xT^d)$ and $n\in\xN$, 
the operator 
$I+J_nG(\eta)$ is invertible from $L^2_n(\xT^d)$ to $L^2_n(\xT^d)$. Moreover, 
the inverse $R_n(\eta)=(I+J_nG(\eta))^{-1}$ satisfies the uniform estimate 
\be\label{N199}
\lA R_n(\eta)\rA_{L^2\to L^2}\le 1,
\ee
and the mapping 
$$
L^2_n(\xT^d)\ni\eta\mapsto R_n(\eta)\in \mathcal{L}(L^2_n(\xT^d))
$$
is of class $C^1$. 
In addition, there exists a non-decreasing function 
$\mathcal{F}\colon [0,+\infty)\to[0,+\infty)$ such that, for all $\eta\in C^\infty(\xT^d)$, all $n\in\xN$ and $u\in L^2_n(\xT^d)$, and all 
$\mu\in [-1/2,1/2]$ there holds
\be\label{N198}
\lA R_n(\eta)u\rA_{H^{\mu}}\le 
\mathcal{F}(\lA \nabla\eta\rA_{L^\infty})\lA u\rA_{H^\mu}.
\ee

\end{lemm}
\begin{proof}
Let $u\in L^2_n(\xT^d)$. Since $J_nG(\eta)u$ belongs to $L^2_n(\xT^d)$,
by the definition of $J_n$,
we see that  $u+J_nG(\eta)u\in L^2_n(\xT^d)$. In addition, if 
$u+J_nG(\eta)u=0$ then $u=0$ since
\be\label{N200}
\int_{\xT^d}(u+J_nG(\eta)u)u\dx 
\ge \int_{\xT^d}u^2\dx,
\ee
which follows easily from the positivity of $G(\eta)$ (see~\e{EQ149}) 
and the fact that $J_n^2=J_n$ and $J_n=J_n^*$. 
This means that the operator 
$I+J_nG(\eta)\colon L^2_n(\xT^d)\to L^2_n(\xT^d)$ is injective, and hence
bijective, as $L^2_n(\xT^d)$ is a finite-dimensional vector space.

Now, by using 
the Cauchy-Schwarz inequality, it follows from \e{N200} that 
$$
\lA u\rA_{L^2}\le \lA u+J_nG(\eta)u\rA_{L^2},
$$
which implies the wanted result \e{N199}. Then, since 
$$L^2_n(\xT^d)\ni\eta\mapsto J_n G(\eta)\in \mathcal{L}(L^2_n(\xT^d))
$$ is of class $C^1$ by Proposition~\ref{P:shape}, we see that $
L^2_n(\xT^d)\ni\eta\mapsto R_n(\eta)\in \mathcal{L}(L^2_n(\xT^d))
$ is also of class $C^1$.

To prove \e{N198}, we apply a variant of the previous argument which consists in observing that, for all functions
$u\in L^{2}_n(\xT^d)$ with the mean $0$,
$$
\int_{\xT^d}(u+J_nG(\eta)u) (G(\eta)u)\dx 
\ge \int_{\xT^d} u G(\eta)u \dx 
\ge \frac{K}{1+\lA \nabla\eta\rA_{L^\infty}}\lA u\rA_{H^{\mez}}^2,
$$
where we used \e{N01}. On the other hand, 
we have
\begin{align*}
\int_{\xT^d}(u+J_nG(\eta)u) G(\eta)u\dx 
&\le \lA u+J_nG(\eta)u\rA_{H^{\mez}}\lA G(\eta)u\rA_{H^{-\mez}}\\
& \le C(\lA \nabla\eta\rA_{L^{\infty}})\lA u+J_nG(\eta)u\rA_{H^{\mez}}\lA u\rA_{H^{\mez}},
\end{align*}
where we used~\e{EQ122} to get the last inequality. 
By combining the previous estimates, we conclude that
$$
\lA u\rA_{H^{\mez}}^2\le \frac{1+\lA \nabla\eta\rA_{L^\infty}}{K} C(\lA \nabla\eta\rA_{L^{\infty}})\lA u+J_nG(\eta)u\rA_{H^{\mez}}\lA u\rA_{H^{\mez}}
$$
which immediately implies the uniform estimate
\be\label{N197}
\lA (I+J_nG(\eta))^{-1}u\rA_{H^{\mez}}\le 
\mathcal{F}(\lA \nabla\eta\rA_{L^\infty})\lA u\rA_{H^\mez},
\ee
for all $u\in L^{2}_n(\xT^d)$ with the mean $0$.
Since $G(\eta)u=0$ if $u$ is a constant
(the proof of this identity is recalled below in the proof of Lemma~\ref{L02}),
the inequality \eqref{N197} holds for constants as well (with $\mathcal{F}$ replaced by~$1$). Now,
for the same reason,
the orthogonal spaces of constants and functions with mean zero
are invariant for $I+J_n G(\eta)$,
and thus they are also invariant for the inverse.
Therefore, \eqref{N197}
holds for
all
$u\in L_{n}^{2}(\mathbb{T}^{d})$.

Now, note that on $L^2_n(\xT^d)$ the operator 
$I+J_nG(\eta)$ coincides with the self-adjoint operator 
$I+J_nG(\eta)J_n$. Therefore, by an elementary duality argument, the estimate \e{N197} implies 
\e{N198} with $\mu=-1/2$. The general case $\mu\in [-1/2,1/2]$ then follows by interpolation.
\end{proof}

Once the previous lemma is granted, we are in a position to rigorously justify the computations leading to~\e{EQ144}. Indeed, directly from the definition of $\tilde A_n(f)$, we verify that $\tilde A_n(f)u\in L^2_n(\xT^d)^3$ for any $u\in L^2_n(\xT^d)^3$, and hence $\tilde A_n(f)(F_n(f))$ belongs to $L^2_n(\xT^d)^3$ for any $f\in L^2_n(\xT^d)^3$. This implies that it is legitimate to apply
$R_n(\eta)$ to this term.
We are thus reduced to solving the Cauchy problem
  \begin{equation}\label{A3}
  \left\{
  \begin{aligned}
  &\partial_t f=
  \mathcal{F}_n(f),\\
  & f\arrowvert_{t=0}=J_n f_{0},
  \end{aligned}
  \right.
  \end{equation} 
where we denoted by $f_0=(\eta_0,\psi_0,\theta_0)$ the initial data.

\begin{prop}\label{P:theta}
Let $n\in\mathbb{N}\setminus\{0\}$. For every $f_0=(\eta_0,\psi_0,\theta_0)^{T}\in L^{2}(\xT^{d})^3$,
there exists $T_n>0$ such that 
the Cauchy problem~\e{A3} has a unique maximal solution
\begin{equation*}
f_n=(\eta_n,\psi_n,\theta_n)^{T}\in C^{1}\big([0,T_n);L^2_n(\xT^{d})^3\big).
\end{equation*}
Moreover, we have the alternative:
\be\label{A4}
T_n=+\infty\qquad\text{or}\qquad \limsup_{t\rightarrow T_n} \lA f_n(t)\rA_{L^2}=+\infty.
\ee
\end{prop}

\begin{proof}
Note that the operator $J_n$ is smoothing; namely, it is a bounded
operator from $H^r(\xT^d)$ to $H^\mu(\xT^d)$ for all real~$r$ and $\mu$. As a result, it follows from Lemma~\ref{L:Rn}, along with Propositions~\ref{P:DN2} and~\ref{P01}, that 
the operator $f\mapsto \mathcal{F}_n(f)$ is locally Lipschitz from $L^{2}_n(\xT^d)^3$ to itself. 
Therefore, for all $n\in\xN\setminus\{0\}$,~\eqref{A3} is in fact an ODE 
with values in the Banach space $L^2_n(\xT^d)^3$, equipped with the usual $L^2$-norm. Consequently, the Cauchy-Lipschitz theorem gives 
the existence of a unique maximal solution~$f_n\in C^{1}([0,T_n);L^{2}_n(\xT^d))^3$.
The alternative \e{A4} is then a consequence of the usual 
continuation principle for ordinary differential equations.
\end{proof}

\subsection*{Step 2: returning to the original equation.} 
In the previous step, we have solved the augmented system~\eqref{A3}. We now prove that these solutions provide solutions to the original equations.

\begin{lemm}
Let $n\in\xN\setminus\{0\}$, and consider a solution
\begin{equation*}
f=(\eta,\psi,\theta)^{T}\in C^{1}\big([0,T_n);L^2_n(\xT^{d})^3\big)
\end{equation*}
to the Cauchy problem~\e{A3}. If at the initial time $t=0$ there holds 
$$
\theta_0=J_n G(\eta_0)\psi_0,
$$
then $\theta=\partial_t\eta$, for all $t\geq0$. In particular, $(\eta,\psi)$ is then a solution 
of the original Cauchy problem~\e{A1}.
\end{lemm}

\begin{proof}
Using the time derivative formula~\e{n:time} for the derivative of the
Dirichlet-to-Neumann map, we write
\begin{align*}
  \begin{split}
  &\partial_{t} (\theta-\partial_t \eta)=\partial_{t} (\theta-J_n G(\eta)\psi)\\
&\indeq
=\partial_{t}\theta-J_n\bigl(G(\eta)(\partial_{t}\psi-B \partial_{t}\eta)\bigr)
+J_n \cnx (V\partial_{t}\eta)\\
&\indeq
= - \Lambda_n(\eta,\psi)\partial_{t}\eta
+ J_n G(\eta) \partial_{t}\psi
- J_n G(\eta) \partial_{t}\psi
 + J_n G(\eta)(B\partial_{t}\eta)
     + J_n \cnx (V\partial_{t}\eta)
  \\
  &\indeq=0,
  \end{split}
\end{align*}
where in the last step we
canceled the second and third terms and
used the definition~\eqref{EQ02}.
Since $\theta(0)-\partial_{t}\eta(0)=0$ by assumption, we obtain
$\theta=\partial_{t}\eta$ for all~$t\geq0$.
\end{proof}

From now on, given $\eta_0$ and $\psi_0$, we fix $\theta_0$ such that
$$
\theta_0=J_n G(\eta_0)\psi_0.
$$
Then the previous lemma implies that the solution $f_n=(\eta_n,\psi_n,\theta_n)^T$ given by Proposition~\ref{P:theta} 
satisfies $\theta_n=\partial_t\eta_n$.

Let us pause now to show that the mean values
of $\eta_n(t,\cdot)$ and $\psi_n(t,\cdot)$ vanish.

\begin{lemm}
\label{L02}
For all $n\in \xN$ and all $t\ge 0$,
\be\label{mean0}
\int_{\xT^d}\eta_n(t,x)\dx=0\quad\text{and}\quad
\int_{\xT^d}\psi_n(t,x)\dx=0.
\ee
\end{lemm}

\begin{proof}
This is proven in~\cite{AZ-2023virial}, while here we provide another proof based on the identity~\eqref{EQ31} for~$N(\eta_n,\psi_n)$. Note that $J_0 f=(2\pi)^{-d}\int f$.
For $\eta_n$, we have
  \begin{align}
  \begin{split}
   \partial_{t}J_0\eta_n
   &=
   J_0 J_n G(\eta_n) \psi_n
   = J_0 G(\eta_n) \psi_n
   = (2\pi)^{-d}\int G(\eta_n) \psi_n\dx
   = 0
   ,
  \end{split}
   \label{EQ14}
  \end{align}
where in the second equality we used
$J_0 J_n f = J_0 f$ for all $n\in\mathbb{N}$. 
It remains to justify the last step in \eqref{EQ14}.
Let $\phi_n$ be the harmonic function corresponding to
$\psi_n$ and $\eta_n$. For every $y_0\leq - C \Vert \eta_n\Vert_{W^{1,\infty}}$, we have
  \begin{align}
  \begin{split}
   &
   \int_{\mathbb{T}} G(\eta_n) \psi_n
   = -\int_{\mathbb{T}}  \partial_y \phi (\cdot, y_0)\dx
  \end{split}
   \llabel{EQ07}
  \end{align}
by the divergence theorem. The desired assertion $\int G(\eta_n)\psi_n\dx=0$ then
follows by sending $y_0\to-\infty$ and applying \eqref{decaytozero}.

From \eqref{EQ14} and
$\int \eta_0 = 0 $, we then
obtain $\int \eta_n=0$, for all~$t\ge0$.

For $\psi_n$, we have
  \begin{align}
  \begin{split}
   \partial_{t} J_0\psi_n
    &=
   -
   J_0
   J_n \bigl(
      N(\eta_n,\psi_n) 
      +L(\eta_n)
       \bigr)
  \\&
  =
  - (2\pi)^{-d}
    \int
      N(\eta_n,\psi_n) 
  - (2\pi)^{-d}
    \int L(\eta_n)
  = 0
  ,
  \end{split}
   \label{EQ15}
  \end{align}
where we employed \eqref{EQ31} in the last step (this is legitimate since $L^2_n(\xT^d)\subset H^s(\xT^d)$ for any $s\geq0$). 
Using \eqref{IEQ90}, we thus obtain
$\int \psi_n=0$ for all~$t\geq0$.
\end{proof}

\subsection*{Step 3: conservation of the 
energy} 

The next step is to prove the following result.

\begin{lemm}\label{L:energy}
Consider an integer $n\in \xN\setminus \{0\}$, a time $T>0$, and a solution 
$(\eta,\psi)\in C^{\infty}\big([0,T];L^2_n(\xT^{d})^2\big)$ 
to the Cauchy problem~\e{A1}. Then 
\begin{equation} \label{EQ03}
\fract\mathcal{E} = 0    ,
\end{equation}
where $\mathcal{E}$ is defined in~\eqref{EQ91}:
\begin{equation}
     2\mathcal{E}(\eta,\psi)
     =
     \int_{\xT^d}
     \psi G(\eta)\psi \dx
       + \int_{\xT^d} (\partial_{t}\eta)^{2} \dx
       + \int_{\xT^d} (\Delta\eta )^{2} \dx
       + g \int_{\xT^d} \eta^{2} \dx
       .
   \label{EQ04}
  \end{equation}
\end{lemm}
\begin{proof}
To obtain the identity~\e{EQ03}, we first
multiply the first equation in \eqref{A1} by $-\partial_{t}\psi$, the
second by $\partial_{t}\eta$, integrate in the $x$ variable on
$\xT^d$, and then add the two resulting equations to obtain
\begin{multline} \label{Energy}
    \frac{1}{2} \fract \int_{\xT^d} (\partial_t \eta)^{2} \dx
     +    \int_{\xT^d}J_{n}G(\eta)\psi \, \partial_t\psi \dx
     \\
     + \int_{\xT^d} J_{n} N \partial_t \eta\dx
     + \int_{\xT^d} J_{n}L(\eta) \partial_t \eta  \dx  = 0
    .
     \end{multline}
Denote the second, third, and fourth terms by $I_1$, $I_2$, and~$I_3$.

Starting with $I_1$, we write
  \begin{align*}
    \begin{split}
    I_{1} & = \fract \int_{\xT^d} G(\eta)\psi \, J_{n} \psi \dx
       -\int_{\xT^d}  \fract ( G(\eta)\psi) \, J_{n}\psi \dx  \\
       & =  \fract \int_{\xT^d} \psi G(\eta)\psi \dx
      - \int_{\xT^d}  \psi  G(\eta)(\partial_t \psi- B \partial_t \eta)   \dx  + \int_{\xT^d}   \cnx_{x} (V \partial_t \eta)  \, \psi \dx 
  ,
  \end{split}
    \end{align*}  
where we used the time derivative formula \eqref{n:time}
for the Dirichlet-to-Neumann map  and the fact $J_{n} \psi= \psi$.
We next integrate by parts in the last term and replace $V$ with its expression \eqref{defi:BV2} to get
 \begin{align*}
   \begin{split}
    I_{1} &   =  \fract \int_{\xT^d} \psi G(\eta)\psi \dx
         - \int_{\xT^d}    \psi G(\eta)(\partial_t \psi- B \partial_t \eta)  \dx
        \\&\indeq
   - \int_{\xT^d}  \partial_t \eta ( \nabla \psi -B \nabla \eta )  \cdot \nabla \psi \dx  \\
  & =  \fract \int_{\xT^d} \psi G(\eta)\psi\dx
        - \int_{\xT^d} \psi  G(\eta)\partial_t\psi \dx
    \\&\indeq
        + \int_{\xT^d} \psi  G(\eta)(B \partial_t \eta) \dx  -   \int_{\xT^d}  \partial_t \eta ( \nabla \psi -B \nabla \eta )  \cdot \nabla  \psi \dx
  .
  \end{split}
\end{align*}  
Note that by self-adjointness of the 
Dirichlet-to-Neumann map on $L^{2}(\xT^d)$, we have
$\int_{\xT^d}   \psi G(\eta)\partial_t \psi\dx = I_{1}$, and thus we  get
  \begin{align*}
    \begin{split}
    I_{1} &   =  \frac{1}{2}\fract \int_{\xT^d} \psi G(\eta)\psi \dx + \frac{1}{2} \int_{\xT^d} \psi  G(\eta) ( B \partial_t \eta) \dx
    \\&\indeq
    -  \frac{1}{2}\int_{\xT^d}  \partial_t \eta ( \nabla \psi -B \nabla \eta )  \cdot \nabla\psi \dx  \\
  & =  \frac{1}{2}\fract \int_{\xT^d} \psi G(\eta)\psi \dx  + \frac{1}{2}\int_{\xT^d} \psi G(\eta)(B \partial_t \eta)  \dx
   \\&\indeq
   -  \frac{1}{2} \int_{\xT^d}  |\nabla \psi|^{2} \partial_t \eta \dx + \frac{1}{2} \int_{\xT^d} B \nabla \eta \cdot  \nabla  \psi \, \partial_t \eta  \dx
  .
  \end{split}
 \end{align*}  
The term $I_{2}$ (the third term in \eqref{Energy}) may be expressed 
using the identity \eqref{N2=} for the nonlinear term $N$ as
   \begin{align*}
     \begin{split}
I_{2}    & = -\frac{1}{2} \int_{\xT^d} J_{n}B (G(\eta) \psi + \nabla \psi 
\cdot \nabla \eta) \partial_t \eta   \dx +   \frac{1}{2} \int_{\xT^d}  J_{n} (| \nabla  \psi |^{2})  \partial_t \eta \dx \\
    &= -\frac{1}{2} \int_{\xT^d} \psi G(\eta) (B \partial_t \eta) \dx - \frac{1}{2}\int_{\xT^d} B  \nabla \psi 
    \cdot \nabla \eta \,  \partial_t \eta \dx  +  \frac{1}{2} \int_{\xT^d}   | \nabla \psi |^{2}  \partial_t \eta \dx 
  ,
  \end{split}
 \end{align*}
where we again used the self-adjointness of the Dirichlet-to-Neumann map $G(\eta)$ on~$L^{2}(\xT^d)$.
Adding the expressions for $I_{1}$ and $I_{2}$, we see that
 \begin{align}
   \begin{split}
I_{1} + I_{2} =   \frac{1}{2}\fract \int_{\xT^d} 
\psi  G(\eta)\psi \dx  
  .
  \end{split}
   \label{EQ103}
\end{align} 
As for the term $I_{3}$, we use \eqref{EQ125} to derive
\begin{align}
  \begin{split}
I_{3} = \frac{1}{2} \fract \int_{\xT^d} (\Delta  \eta)^{2} \dx +  \frac{g}{2} \fract \int_{\xT^d} \eta^{2} \dx
,
  \end{split}
   \label{EQ104}
\end{align}
and then \eqref{Energy}, \eqref{EQ103}, and \eqref{EQ104} lead
to~\eqref{EQ03}.
\end{proof}

\subsection*{Step 4: the global existence} 

To conclude the proof of Proposition~\ref{T2}, it only remains to prove that
\begin{equation*}
T_n=+\infty\quad \text{for all} \quad n\in\mathbb{N}.
\end{equation*}
In light of 
the dichotomy \eqref{A4}, 
the previous claim 
is proven if we establish an estimate for  
$\lA (\eta_n,\psi_n,\theta_n)(t)\rA_{L^2}$ which is uniform in $n$ and in $t$. 
To do this, recall that
  \begin{equation*}
   \mathcal{E}(\eta_n,\psi_n)(t)=
\mathcal{E}(\eta_n,\psi_n)(0)
    \comma t\in [0,T_n)
    ,
  \end{equation*}
where $\mathcal{E}(\eta,\psi)$ is defined in~\eqref{EQ04}.

Since $\theta_n=\partial_t\eta_n$, 
we have
$$
\lA \theta_n\rA_{L^2}^2\le 2\mathcal{E}(\eta_n,\psi_n).
$$
On the other hand, for any $g\ge 0$, we also have an elementary estimate
$$
2\mathcal{E}(\eta_n,\psi_n)\ge \int_{\xT^d} (\Delta\eta_n )^{2} \dx
       + g \int_{\xT^d} \eta_n^{2} \dx \ge  \lA \eta_n\rA_{H^2}^2\ge  \lA \eta_n\rA_{L^2}^2.
$$
Note that the above inequality also holds for $g=0$ since the mean-value of $\eta_n$ vanishes for all time (see Lemma~\ref{L02}). 

It remains to control the $L^2$-norm of $\psi_n$. For this, we will apply a non-trivial variant of the trace inequality, and this is where the assumption $1\le d\le 2$ enters. 
Namely, we use the following inequality
(see~\cite{A-Berkeley,AZ-2023virial}):
For any $d\ge 1$, there exists $c=c(d)$ such that, for all $\eta\in W^{1,\infty}(\xT^d)$ and $\psi\in H^{1/2}(\xT^d)$, we have
\be\label{Ntrace}
\int_{\xT^d}
     \psi G(\eta)\psi \dx
     \ge \frac{c}{1+\lA \nabla\eta\rA_{{\rm BMO}(\xT^d)}}\lA \psi\rA_{\dot{H}^{1/2}(\xT^d)}^2,
\ee
where $c>0$ and
${\rm BMO}(\xT^d)$ stands for the space of functions with bounded mean oscillations. In light of the classical embedding 
$H^1(\xT^d)\subset {\rm BMO}(\xT^d)$ for $d\le 2$, we get 
$$
d\le 2 ~\Rightarrow ~2\mathcal{E}(\eta_n,\psi_n)\ge \int_{\xT^d}
     \psi_n G(\eta_n)\psi_n \dx
     \ge \frac{c'}{1+\lA \eta_n\rA_{H^2(\xT^d)}}\lA \psi_n\rA_{\dot{H}^{1/2}(\xT^d)}^2,
$$
where $c'>0$.
Then, it follows from the conservation of energy and 
the uniform estimate of the $H^2$-norm of $\eta_n$ that
$$
\lA \psi_n\rA_{\dot{H}^{1/2}(\xT^d)}^2
\lec \frac{1+\mathcal{E}(\eta_n,\psi_n)^{1/2}}{c'}\mathcal{E}(\eta_n,\psi_n).
$$
Now, since the mean-value of $\psi_n$ is zero (see Lemma~\ref{L02}),
this immediately implies a uniform bound for the $L	^2$-norm of
$\psi_n$, which completes the proof of Proposition~\ref{T2}.

\section{Proof of the main theorem}

From here on, we assume that $d=1$. 
Also, recall from \eqref{IEQ90} that the mean values of $\eta_0$ and
$\psi_0$ vanish.

\subsection{Uniform bounds.} 
We begin by proving 
that the approximate solutions $(\eta_{n},\psi_{n})$ satisfy the
necessary uniform bounds that will allow 
us to pass to the limit using the Arzela-Ascoli theorem. 

Hereafter, given a function space $X\subset H^{-2}(\xT)$, we use $L^\infty X$ as a compact notation for $L^\infty([0,+\infty);X)$.

Directly from the definition~\e{EQ91} of the energy and from the energy conservation property \eqref{EQ03}, we have 
\be\label{N00}
\mathcal{E}(\eta_n,\psi_n)(t)=\mathcal{E}(\eta_n,\psi_n)(0)
,
\ee
where we also recall~\eqref{EQ04}.

We begin by repeating arguments used in the last step of the previous section. 
Here again we use the classical trace inequality (see~\e{EQ149}) which reads
\begin{equation}
\int_{\mathbb{T}}
     \psi G(\eta)\psi \dx\ge C(\lA\partial_x\eta\rA_{L^\infty})\lA \psi\rA_{\dot{H}^\mez(\xT)}^2
     ,
\label{EQ149b}
\end{equation}
where $\lA \psi\rA_{\dot{H}^\mez(\xT)}=\blA \psi-\hat{\psi}(0)\brA_{H^\mez(\xT)}$.

Then, we begin by listing some uniform estimates for~$\eta_n$. 
As a trivial consequence of
\eqref{N00},
we immediately infer that
\begin{equation*}
\sup_{n\in \xN}\lA \eta_n\rA_{L_t^{\infty}H_x^2}<+\infty.
\end{equation*}
In light of the Sobolev embedding, this implies
\begin{equation}
\sup_{n\in \xN}\lA \partial_x\eta_n\rA_{L_t^{\infty}L_x^\infty}<+\infty,
\label{N02}
\end{equation}
and hence
\eqref{EQ149b}
and
$\int_{\mathbb{T}}\psi_nG(\eta_n)\psi_n\leq \frac12 \mathcal{E}(\eta_n,\psi_n)$ 
imply
\begin{equation}\llabel{N03}
\sup_{n\in\xN}\lA \psi_n\rA_{L_t^{\infty}\dot{H}_x^\mez}<+\infty.
\end{equation}
Note that, since the mean value of $\psi_n(t,\cdot)$ vanishes according to Lemma~\ref{L02}, we can always replace homogeneous Sobolev norms 
by inhomogeneous ones. This gives
\begin{equation}\label{N04}
\sup_{n\in\xN}\lA \psi_n\rA_{L_t^{\infty}H_x^\mez}<+\infty.
\end{equation}
However, one may prove a stronger result.

\begin{lemm}
There holds
\begin{equation}\label{N05i}
\sup_{n\in\xN}\lA \psi_n\rA_{L_t^{\infty}H_x^{1}}<+\infty.
\end{equation}
\end{lemm}
\begin{proof}
The main idea is to use 
the Rellich inequality, that is the following estimate (see~\cite{McLean,Necas}):
\begin{equation}
\label{Rellich}
\Vert \partial_{x}\psi\Vert_{L^2}\leq C(\Vert \partial_{x}\eta\Vert_{L^\infty})\Vert G(\eta)\psi\Vert_{L^2}.
\end{equation}
Indeed, in light of \e{Rellich}, 
the conservation of the mean value of $\psi_n$
(see~\eqref{mean0}), and the uniform control of the Lipschitz norm of $\eta_n$ (see~\eqref{N02}), we see that to prove \e{N05i}, it suffices to 
establish the inequality
\begin{equation}\label{N05}
\sup_{n\in\mathbb{N}}\lA G(\eta_n)\psi_n \rA_{L_t^{\infty}L_x^{2}}<+\infty.
\end{equation}
Recall that the first equation in \eqref{A1} takes the form
$
\partial_t\eta_n=J_nG(\eta_n)\psi_n
$,
and thus, directly from the conservation of energy (see~\eqref{N00}), we obtain
\begin{equation}\label{EQ06.5}
\sup_{n\in\mathbb{N}}\lA J_n G(\eta_n)\psi_n \rA_{L_t^{\infty}L_x^{2}}<+\infty.
\end{equation}
Consequently, to achieve the desired uniform estimate \eqref{N05}, it is sufficient to prove that
\begin{equation}\label{N11}
\sup_{n\in\mathbb{N}}\lA (I-J_n)G(\eta_n)\psi_n \rA_{L_t^{\infty}L_x^{2}}<+\infty.
\end{equation}
We aim to establish \eqref{N11} by leveraging the fact that $(I-J_n)\psi_n=0$ by construction. This requires commuting $I-J_n$ with $G(\eta_n)$, which presents two challenges: The projection $I-J_n$ is a Fourier multiplier with a non-smooth symbol, and $G(\eta_n)$ is determined by solving a boundary value problem. To address these difficulties, we turn to the paradifferential analysis of the Dirichlet-to-Neumann operator. Specifically, we utilize Proposition \ref{coro:paraDN1}, which allows a comparison between $G(\eta_n)$ and $G(0)=\la D_x \ra$.
Recall that this proposition asserts that the remainder term
\begin{equation*}
R(\eta_n):=G(\eta)\psi-\la D_x\ra \psi,
\end{equation*}
satisfies
\begin{equation}\label{N12}
    \left\| R(h)f \right\|_{H^{\sigma-1+\epsilon}}\le \mathcal{F}\left(\| h \|_{H^{s}}\right)\left\| f \right\|_{H^{\sigma}},
\end{equation}
provided 
\begin{equation*}
s>\frac{3}{2},\quad \frac{1}{2}\leq \sigma \leq s-\frac{1}{2},\quad 
0<\epsilon\leq \frac{1}{2}, \quad \epsilon< s-\frac{3}{2}.
\end{equation*}
Now, write
\begin{equation*}
(I-J_n) G(\eta_n)\psi_n=(I-J_n)\la D_x\ra \psi_n+(I-J_n)R(\eta_n)\psi_n
,
\end{equation*}
and observe that, since the operators $I-J_n$ and $\la D_x \ra$ are Fourier multipliers, they commute, leading to
\begin{equation*}
(I-J_n)\la D_x\ra \psi_n=0.
\end{equation*}
Thus, we arrive at
\be\label{N13}
(I-J_n) G(\eta_n)\psi_n=(I-J_n)R(\eta_n)\psi_n.
\ee
We will exploit this relation to infer \eqref{N11} through a two-step argument. First, we will establish a uniform bound in $L^\infty_t(H^{-\epsilon}_x)$ for some~$\epsilon>0$. Subsequently, we will employ a bootstrap argument to derive \eqref{N11} from this intermediate step.

We begin by using~\e{N12} 
with
\begin{equation*}
s=2,\quad \sigma=\mez,\quad \eps=\frac{1}{4}.
\end{equation*}
Then $\sigma-1+\eps=-1/4$ and we have
\begin{equation*}
\lA (I-J_n)R(\eta_n)\psi_n\rA_{H^{-\frac{1}{4}}}\le 
\lA R(\eta_n)\psi_n\rA_{H^{-\frac{1}{4}}}\le 
C(\lA \eta_n\rA_{H^2})\lA \psi_n\rA_{H^{\mez}}.
\end{equation*}
In particular, using~\e{N04}, we get
\begin{equation}\label{N15}
\sup_{n\in\mathbb{N}}\Big(\lA (I-J_n)R(\eta_n)\psi_n\rA_{L_t^\infty H_x^{-\frac{1}{4}}}+ 
\lA R(\eta_n)\psi_n\rA_{L_t^\infty H_x^{-\frac{1}{4}}}\Big)<+\infty.
\end{equation}
Now, remembering \e{EQ06.5} and \e{N13}, we deduce from 
the previous estimate that
\begin{equation}\label{N14}
\sup_{n\in\mathbb{N}}\lA G(\eta_n)\psi_n \rA_{L_t^{\infty}H_x^{-\frac{1}{4}}}<+\infty.
\end{equation}
Then, by writing 
\begin{equation*}
\la D_x\ra\psi_n=G(\eta_n)\psi_n-R(\eta_n)\psi_n,
\end{equation*}
we see that \e{N15} and \e{N14} imply 
\begin{equation*}
\sup_{n\in\mathbb{N}}\lA \la D_x\ra\psi_n \rA_{L_t^{\infty}H_x^{-\frac{1}{4}}}<+\infty.
\end{equation*}
Using again~\e{mean0}, we infer that
\begin{equation}\label{N17}
\sup_{n\in\mathbb{N}}\lA \psi_n \rA_{L_t^{\infty}H_x^{\frac{3}{4}}}<+\infty.
\end{equation}

Next, we repeat the previous arguments, but now 
using the uniform control of the $L^\infty_t(H^{3/4}_x)$ norm from \e{N17} instead of the bound~\e{N04}. 
In this second step, we use the estimate~\e{N12} 
with
\begin{equation*}
s=2,\quad \sigma=\frac{3}{4},\quad \eps=\frac{1}{4}.
\end{equation*}
Then $\sigma-1+\eps=0$, and we have
\begin{equation*}
\lA (I-J_n)R(\eta_n)\psi_n\rA_{L^2}\le 
\lA R(\eta_n)\psi_n\rA_{L^2}\le 
C(\lA \eta_n\rA_{H^2})\lA \psi_n\rA_{H^{\frac{3}{4}}}.
\end{equation*}
Then, in parallel to what we have done above, we successively deduce that
\begin{equation*}
\sup_{n\in\mathbb{N}}\Big(\lA (I-J_n)R(\eta_n)\psi_n\rA_{L_t^\infty L_x^2}+ 
\lA R(\eta_n)\psi_n\rA_{L_t^\infty L_x^2}\Big)<+\infty,
\end{equation*}
and
\begin{equation}\label{N0202}
\sup_{n\in\mathbb{N}}\lA G(\eta_n)\psi_n \rA_{L^{\infty}L^2}<+\infty.
\end{equation}
Again, by writing 
$\la D\ra\psi_n=G(\eta_n)\psi_n-R(\eta_n)\psi_n$, 
we infer that
\begin{equation}\llabel{EQ100}
\sup_{n\in\mathbb{N}}\lA \la D_x\ra\psi_n \rA_{L^{\infty}L^2}<+\infty,
\end{equation}
which in turn implies the wanted result~\e{N05},
completing the proof.
\end{proof}

\begin{rema}
Although we were guided by the Rellich estimate~\eqref{Rellich}, we did not use it directly. Instead, we relied on~\e{N12} which is a more precise version of~\e{Rellich}, valid when $\eta$ is smoother than Lipschitz.
\end{rema}
Now, set 
\begin{equation*}
B_n\defn \frac{
       G(\eta_n) \psi_n + \partial_{x}\eta_n \partial_{x}\psi_n
         }{
      1+(\partial_x \eta_n)^2
     }
\end{equation*}
and
\begin{equation*}
   V_n\defn  \partial_{x}\psi_n - B_n \partial_{x}\eta_n.
\end{equation*}
The previous 
estimates~\e{N02},~\e{N05i}, and 
\e{N0202} imply directly that
  \begin{equation} \label{EQ12}
\sup_{n\ge 1}\big(\lA B_n\rA_{L_t^{\infty}L_x^{2}}+\lA V_n\rA_{L_t^{\infty}L_x^{2}}\big)<+\infty.
\ee
Now, in order to get some compactness through the Arzela--Ascoli theorem, we need information on time derivatives. As we will see, it is straightforward to bound $\partial_t\eta_n$, while, by contrast, it is much more difficult to estimate $\partial_t\psi_n$. To overcome this problem, we will proceed by an indirect argument: First, we will estimate the time derivative of the unknown $u_n\defn \psi_n+J_nG(\eta_n)\psi_n$. This will give some compactness on $(u_n)$. By combining this with the compactness of $(\eta_n)$, we will deduce some compactness for $(\psi_n)$ as well. 

\begin{lemm}
The following uniform estimates hold:
 \begin{align}
&\sup_{n\ge 1}\lA
   \partial_{t}\eta_n\rA_{L_t^\infty L_x^2}
    <+\infty,
   \label{EQ13}\\
&\sup_{n\ge 1}\lA
   \partial_{t}(\psi_n+J_nG(\eta_n)\psi_n)\rA_{L_t^{\infty}H_x^{-2}}
    <+\infty.
   \label{EQ13bis}
\end{align}
\end{lemm}
\begin{proof}
The uniform estimate~\e{EQ13} 
is obtained directly from the conservation of energy. To estimate the time derivative of 
$\psi_n+J_nG(\eta_n)\psi_n$, we use the equations for $\eta_n$ and $\psi_n$ (see~\e{A1}) to write
\begin{align*}
\partial_{t} (\psi_n+J_nG(\eta_n)\psi_n)
  =
   - J_n N(\eta_n,\psi_n)
   - J_n L \eta_n.
  \end{align*}
By Propositions~\ref{propN} and~\ref{P02}, we have
$$
N(\eta_n,\psi_n)=\frac12 V_n^2 - \frac12 B_n^2 + B_n V_n \partial_x\eta_n.
$$
Then by 
using the control of the $L^\infty_tL^\infty_x$-norm of $\partial_x\eta_n$ (see~\e{N02}), and the $L_t^\infty L_x^2$-norm of $V_n$ and $B_n$ (see~\e{EQ12}), we deduce from the embedding $L^1(\xT)\subset H^{-1}(\xT)$ that
\begin{equation*}
\sup_{n\ge 1}\lA N(\eta_n,\psi_n)\rA_{L_t^{\infty}H_x^{-1}}
\lec \sup_{n\ge 1}\lA \frac12 V_n^2 - \frac12 B_n^2 + B_n V_n \partial_{x}\eta_n\rA_{L_t^{\infty}L_x^{1}}<+\infty.
 \end{equation*}
Finally, it is immediate that
  \begin{equation*}
\sup_{n\ge 1}\lA
  J_n L(\eta_n)\rA_{L_t^\infty H_x^{-2}} \le  \sup_{n\ge 1}\lA g \eta_n + \partial_{x}^{4} \eta_n
 \rA_{L_t^{\infty}H_x^{-2}}<+\infty,
  \end{equation*}
since $(\eta_n)$ is bounded in $L_t^{\infty}H_x^{2}$.\end{proof}

\subsection{Passing to the limit.}

Introduce the sequence
$$
u_n\defn \psi_n+J_nG(\eta_n)\psi_n.
$$
We have seen in \eqref{EQ13bis}
that 
$(\partial_tu_n)$ is bounded in 
$L_t^{\infty}H_x^{-2}$. 
Also, in light of \e{ests+2}, $(u_n)$ is bounded in $L_t^\infty L_x^{2}$ since 
$(\psi_n)$ is bounded in $L_t^{\infty}H_x^{1}$ and $(\eta_n)$ is bounded in $L_t^{\infty}H_x^{2}$.

By the Banach--Alaoglu and Arzela--Ascoli theorems, it follows from the previous bounds that there exist three functions $\eta$, $u$, and $\psi$ along with a subsequence
$k=k(n)$ of $1,2,3,\ldots$ such that
  \begin{alignat*}{2}
   &
   \eta_{k} \to \eta
   \qquad &&\text{weakly-$*$ in~}
   L_t^{\infty}H_x^{2}
   \hbox{~and~}
   \text{strongly in~} C_tH_x^{2-\epsilon}
   \text{~for every $\epsilon>0$,}
  \\&
   u_{k} \to u
   \qquad  &&\text{weakly-$*$ in~}
   L_t^{\infty}L_x^{2}
   \hbox{~and~}
   \text{strongly in~} C_tH_x^{-\epsilon}
   \text{~for every $\epsilon>0$}
   .
   \\&
   \psi_{k} \to \psi
   \qquad  &&\text{weakly-$*$ in~}
   L_t^{\infty}H_x^{1}
   .
\end{alignat*}

Now, we claim that 
\be\label{N195}
\psi_{k} \to \psi
   \qquad  
   \text{strongly in~} C_tH_x^{1-\epsilon}
   \text{~for every $\epsilon>0$}
   .
\ee
Since $\psi_k-\psi$ is bounded in $L_t^\infty H_x^1$, it is sufficient to prove that 
$$
\psi_{k} \to \psi
   \qquad  
   \text{strongly in~} C_tH_x^{-1/2}
   .
$$
To prove the latter convergence, we will prove that $(\psi_k)$ is a Cauchy sequence in~$C_tH_x^{-1/2}$. To see this, we will use Lemma~\ref{L:Rn} for the operator
$$
R_n(\eta)=(I+J_nG(\eta))^{-1}.
$$
We begin by writing
\begin{equation}\label{N196}
\begin{aligned}
\psi_n-\psi_p
&=R_n(\eta_n)( I+J_nG(\eta_n))(\psi_n-\psi_p)\\
&=R_n(\eta_n)(u_n-u_p)
-R_n(\eta_n)(J_nG(\eta_n)-J_pG(\eta_p))\psi_p\\
&=R_n(\eta_n)(u_n-u_p)
-R_n(\eta_n)J_n(G(\eta_n)-G(\eta_p))\psi_p\\
&\quad-
R_n(\eta_n)(J_n-J_p)G(\eta_p)\psi_p.
\end{aligned}
\end{equation}
Then since, for any $\eps>0$,
\begin{alignat*}{2}
&(u_n) &&\text{ is a Cauchy sequence in 
$C_tH_x^{-\epsilon}$},\\
&(\eta_n) &&\text{ is a Cauchy sequence in 
$C_tH_x^{2-\epsilon}$},\\
&(\psi_n) &&\text{ is bounded in 
$C_tH_x^{1}$},\\
&(G(\eta_n)\psi_n) &&\text{ is bounded in 
$C_tL_x^{2}$},\\
&(J_n) &&\text{ is a Cauchy sequence in 
$\mathcal{L}(L^2(\xT),H^{-\eps}(\xT))$},
\end{alignat*}
we deduce from~\e{N196}, along with the estimates \e{N197} and \e{EQ33}, that $(\psi_n)$ is a Cauchy sequence in $C_t H_x^{-1/2}$. This completes the proof of the claim~\e{N195}.

Our next step is to prove that
  \begin{equation}
   \int_{-\infty}^{\eta_{k}}
   \partial_{x}\phi_{k}
   \partial_{y}\phi_{k}   
   \to
   \int_{-\infty}^{\eta}
   \partial_{x}\phi
   \partial_{y}\phi   
   \text{~strongly in~} C_t L^{1}
  ,
   \label{EQ27}
  \end{equation}
for which we use Lemma~\ref{L01}. 
More precisely, to prove \eqref{EQ27}, we introduce
  \begin{equation*}
   \tphi_{k}(x,y)
   = \phi_{k}(x,y+(\eta-\eta_{k}))
   .
  \end{equation*}
Then we have
  \begin{align*}
    \begin{split}
  &
   \int_{-\infty}^{\eta}
   \partial_{x}\phi
   \partial_{y}\phi   
   -
    \int_{-\infty}^{\eta_{k}}
   \partial_{x}\phi_{k}
   \partial_{y}\phi_{k}
  \\&\indeq
  = \int_{-\infty}^{\eta}
   \partial_{x}\phi
   \partial_{y}\phi   
   -
    \int_{-\infty}^{\eta}
    \bigl(
     \partial_{x}\tphi_{k}
     - \partial_{y}\tphi_{k}\partial_{x}(\eta-\eta_{k})
    \bigr)
    \partial_{y}\tphi_{k}
    \\&\indeq
    =
   \int_{-\infty}^{\eta}
   (
   \partial_{x}\phi
   \partial_{y}\phi   
   -
   \partial_{x}\tphi_k
   \partial_{y}\tphi_{k} 
   )   
   +
    \int_{-\infty}^{\eta}
     (     \partial_{y}\tphi_{k})^2
     \partial_{x}(\eta-\eta_{k})
  .
  \end{split}
   \end{align*}
Taking the $L^{1}$ norm of the difference, we then get
  \begin{align*}
    \begin{split}
   &
   \left\Vert
   \int_{-\infty}^{\eta}
   \partial_{x}\phi
   \partial_{y}\phi   
   -
    \int_{-\infty}^{\eta_{k}}
   \partial_{x}\phi_{k}
   \partial_{y}\phi_{k}
   \right\Vert_{L^1_x}
   \\&\indeq
   \leq
   \left\Vert   
      \int_{-\infty}^{\eta}
   \partial_{x}\phi
   (
   \partial_{y}\phi   
   -
   \partial_{y}\tphi_{k}   
   )   
   \right\Vert_{L^1_x}
   +
   \left\Vert   
      \int_{-\infty}^{\eta}
   \partial_{y}\tphi_{k}
   (
   \partial_{x}\phi
   -
   \partial_{x}\tphi_{k}   
   )   
   \right\Vert_{L^1_x}
   \\&\indeq\indeq
   +
   \left\Vert
       \int_{-\infty}^{\eta}
     (     \partial_{y}\tphi_{k})^2
   \right\Vert_{L^1_x}
   \Vert \eta-\eta_{k}\Vert_{W^{1,\infty}}
   \\&\indeq
   \leq
   \Vert   
    \partial_{x}\phi
   \Vert_{L^2(y<\eta)}
   \Vert
   \partial_{y}\phi   
   -
   \partial_{y}\tphi_{k}   
   \Vert_{L^2(y<\eta)}
   +
   \Vert   
   \partial_{y}\tphi_{k}
   \Vert_{L^2(y<\eta)}
   \Vert
   \partial_{x}\phi
   - \partial_{x}\tphi_{k}
   \Vert_{L^2(y<\eta)}
   \\&\indeq\indeq
   +
   \Vert   
   \partial_{y}\tphi_{k}
   \Vert_{L^2(y<\eta)}^2
   \Vert \eta-\eta_{k}\Vert_{W^{1,\infty}}
   .
  \end{split}
  \end{align*}

Applying \eqref{EQ41} and \eqref{EQ44} from Lemma~\ref{L01},
we get
  \begin{align*}
    \begin{split}
   &
   \left\Vert
   \int_{-\infty}^{\eta}
   (
   \partial_{x}\phi \partial_{y}\phi
     - \partial_{x}\tphi_{k} \partial_{y}\tphi_{k}
   )
   \right\Vert_{L^1_x}
   \lec
   \Vert \eta-\eta_k\Vert_{W^{1,\infty}}
   +
   \Vert \psi-\psi_k\Vert_{H^{\mez}}
   .
  \end{split}
  \end{align*}
Therefore, since $L^1(\xT)\subset H^{-1}(\xT)$, we have established
  \begin{equation*}
   \sup_{t\geq0}
   \left\Vert
   \partial_{x}
   \int_{0}^{\eta_{k}}
   \partial_{x}\phi_{k}
   \partial_{y}\phi_{k}   
   -
   \partial_{x}
   \int_{0}^{\eta}
   \partial_{x}\phi
   \partial_{y}\phi   
   \right\Vert_{H^{-2}}
   \to 0
   ,
     \end{equation*}
i.e.,
$
\sup_{t\geq0}
   \Vert N_k-N\Vert_{H^{-2}} \to 0
$,
as $k\to\infty$, where $N=\partial_{x} \int_{0}^{\eta}\partial_{x}\phi\partial_{y}\phi$.
Note that the only other nonlinear term in the system is
$G(\eta_{k})\psi_{k}$, for which we write, for a fixed~$t$,
  \begin{align}
  \begin{split}
   &
   \Vert
   G(\eta_{k})\psi_{k}
   -   G(\eta)\psi
   \Vert_{H^{-1/2}}
\\&\indeq
   \le
   \Vert
    G(\eta_{k})\psi_{k}
   -   G(\eta_{k})\psi
   \Vert_{H^{-1/2}}
   +
   \Vert
   G(\eta_{k})\psi
   -   G(\eta)\psi
   \Vert_{H^{-1/2}}
   \\&\indeq  
   \le
   C(\Vert \eta_{k}\Vert_{W^{1,\infty}})
   \Vert \psi_{k}-\psi\Vert_{H^{1/2}}
   +
   C(\lA (\eta_k,\eta)\rA_{W^{1,\infty}})
   \Vert
   \eta_{k}
   -   \eta
   \Vert_{W^{1,\infty}}
   \Vert \psi\Vert_{H^{1/2}}
   \\&\indeq
   \lec  
   \Vert \psi_{k}-\psi\Vert_{H^{1/2}}
   +
   \Vert
   \eta_{k}
   -   \eta
   \Vert_{W^{1,\infty}}
   ,
  \end{split}
   \label{EQ34}
  \end{align}
where we used
Propositions~\ref{P:DN2}
and~\ref{P01} in the second inequality.
Since the last expression in \eqref{EQ34} converges
to~$0$ uniformly in $t$, we get
  \begin{equation*}
   \sup_{t\geq0}
   \Vert
   G(\eta_{k})\psi_{k}
   -
   G(\eta)\psi
   \Vert_{H^{-1/2}}
   \to
   0
   .
  \end{equation*}

Passing then to the limit in the 
system \eqref{A1} and using the convergence facts above, we obtain
that
the limit $(\eta,\psi)$ satisfies the system~\eqref{EQ28}.

\colb
\section*{Acknowledgments}
T.A.~was supported in part by the French ANR project BOURGEONS, while
I.K.~was supported in part by the
NSF grant DMS-2205493.
The research was performed during the program ``Mathematical Problems in Fluid Dynamics, Part~2,'' held during the summer of 2023 by the Simons Laufer Mathematical Sciences Institute (SLMath), which is supported by the National Science Foundation (Grant No.~DMS-1928930).

\end{document}